\documentclass[a4paper,12pt,leqno]{amsart}
\usepackage{latexsym}
\usepackage[all]{xy}

\usepackage{amssymb} 
\usepackage{amsmath} 
\usepackage{amscd}
\usepackage{color}
\usepackage{comment}
\usepackage{mathtools}
\usepackage{fancybox}

\definecolor{gray}{gray}{0.5}

\numberwithin{equation}{section} 

\newtheorem{theorem}{Theorem}[section]
\newtheorem{lemma}[theorem]{Lemma} 
\newtheorem{corollary}[theorem]{Corollary}
\newtheorem{proposition}[theorem]{Proposition} 
 
\newtheorem{remark}[theorem]{Remark}
\newtheorem{example}[theorem]{Example}

\newtheorem*{theoremA*}{Theorem A}
\newtheorem*{theoremB*}{Theorem B}

\allowdisplaybreaks[3]

\def\C{\mathbb C}

\def\Q{\mathbb Q}
\def\Z{\mathbb Z}

\DeclareMathOperator{\rank}{rank}
\DeclareMathOperator{\image}{Im}
\DeclareMathOperator{\Hess}{Hess}

\newcommand{\Emb}[1]{\phi_{#1}}
\newcommand{\PM}{\zeta}
\newcommand{\RP}{\phi}
\newcommand{\Sn}{\mathfrak{S}_n}
\newcommand{\Pet}[1]{Pet_{#1}}
\newcommand{\Perm}[1]{Perm_{#1}}

\begin{document}

\text{}\vspace{-5pt}

\title[The cohomology rings of Peterson varieties]{The integral cohomology rings of \\Peterson varieties in type A}
\author {Hiraku Abe}
\address{Faculty of Science, Department of Applied Mathematics, Okayama University of
Science, 1-1 Ridai-cho, Kita-ku, Okayama, 700-0005, Japan}
\email{hirakuabe@globe.ocn.ne.jp}

\author {Haozhi Zeng}
\address{School of Mathematics and Statistics, Huazhong University of Science and Technology, Wuhan, 430074, P.R. China}
\email{zenghaozhi@icloud.com} 

\begin{abstract}
In this paper, we study the ring structure of the integral cohomology of the Peterson variety of type $\text{A}_{n-1}$.
We give two kinds of descriptions:
(1) we show that it is isomorphic to the $\Sn$-invariant subring of the integral cohomology ring of the permutohedral variety, 
(2) we determine the ring structure in terms of ring generators and their relations.
\end{abstract}

\maketitle

\section{Introduction}\label{sec: intro}

Let $n(\ge2)$ be a positive integer and $Fl_n=Fl(\C^n)$ the flag variety of $\C^n$ which is the collection of nested sequence of linear subspaces of $\C^n$:
\begin{align*}
 Fl_n = \{ V_{\bullet}=(V_1\subset V_2 \subset \cdots \subset V_n=\C^n) \mid \dim_{\C}V_i = i \ (1\le i\le n)\}.
\end{align*}
Let $N$ be an $n\times n$ regular nilpotent matrix viewed as a linear map $N\colon \C^n\rightarrow \C^n$. The Peterson variety (of type A$_{n-1}$) is a subvariety of $Fl_n$ defined by
\begin{align*}
 \Pet{n} \coloneqq \{ V_{\bullet}\in Fl_n \mid NV_j \subseteq V_{j+1} \text{ for all $1\le j<n$}\},
\end{align*}
where $NV_j$ denotes the image of $V_j$ under the map $N\colon \C^n\rightarrow \C^n$. It was introduced by Dale Peterson to study the quantum cohomology ring of $Fl_n$, and it has appeared in several contexts (e.g., \cite{AHMMS,Ba,ha-ho-ma,Kostant,Rietsch06}).

The cohomology ring $H^*(\Pet{n};\C)$ has been studied in Harada-Tymoczo (\cite{HaTy}), Fukukawa-Harada-Masuda (\cite{fu-ha-ma}), and Harada-Horiguchi-Masuda (\cite{ha-ho-ma}). Moreover, a natural basis of $H^*(\Pet{n};\C)$ which has certain positivity and integrality was discovered (\cite{AHKZ,GoGo,GoMiSi,HaTy}), and it is now actively studied in connection with mixed Eulerian numbers in combinatorics (\cite{AHKZ, GoGo, GoMiSi, GoSi, Ho}). In this paper, we study the ring structure of the cohomology of $\Pet{n}$ with $\Z$ coefficients.

To state the first theorem of this paper, we need to introduce a toric variety which is called the permutohedral variety.
Let $S$ be an $n\times n$ regular semisimple matrix 
viewed as a linear map $S\colon \C^n\rightarrow \C^n$ as above.
The permutohedral variety is defined by
\begin{align*}
 \Perm{n} \coloneqq \{ V_{\bullet}\in Fl_n \mid SV_j \subseteq V_{j+1} \text{ for all $1\le j<n$}\}.
\end{align*}
It is known that $\Perm{n}$ is the non-singular projective toric variety associated with the fan consisting of the set of Weyl chambers of type $\text{A}_{n-1}$ (\cite[Theorems 6 and 11]{De Mari-Procesi-Shayman}).
The symmetric group $\Sn$ of $n$-letters permutes the set of Weyl chambers, and hence there is a natural $\Sn$-action on the cohomology ring $H^*(\Perm{n};\Z)$ which preserves the grading and the cup product. This implies that the invariant subgroup $H^*(\Perm{n};\Z)^{\Sn}$ is in fact a graded ring with respect to the cup product.

The first theorem of this paper is the following.

\begin{theorem}\label{thm: A}
As graded rings, we have $H^*(\Pet{n};\Z) \cong  H^*(\Perm{n};\Z)^{\Sn}$.
\end{theorem}

We note that the corresponding claim for cohomology rings with $\C$ coefficients is known as mentioned in \cite[Sect.\ 1]{AHHM} based on the explicit presentations for the rings $H^*(\Pet{n};\C)$ and $H^*(\Perm{n};\C)^{\Sn}$ due to Fukukawa-Harada-Masuda (\cite{fu-ha-ma}) and Klyachko (\cite{Klyachko}), respectively (see also \cite{AHMMS,ha-ho-ma}).

As the second theorem, we give an explicit presentation of the ring structure of $H^*(\Pet{n};\Z)$ in terms of ring generators and their relations. For simplicity, we assume that the regular nilpotent matrix $N$ appearing in the definition of $\Pet{n}$ is in Jordan canonical form.
Let $L_i$ be the $i$-th tautological line bundle over $Fl_n$ $(1\le i\le n)$. By abusing notation, we denote the restriction of $L_i$ over $\Pet{n}$ by the same symbol. Let $\Z[y_1,y_2,\ldots,y_n]$ be the polynomial ring over $\Z$ with indeterminates $y_1,y_2,\ldots,y_n$. We regard this polynomial ring as a graded ring with $\deg y_i=2$ for $1\le i\le n$. Let 
\begin{align*}
 \RP \colon \Z[y_1,y_2,\ldots,y_n] \rightarrow  H^*(\Pet{n};\Z)
\end{align*} 
be the ring homomorphism
which sends $y_i$ to the first Chern class $c_1(L_i^*)$, where $L_i^*$ is the dual line bundle of $L_i$ ($1\le i\le n$).
We introduce the following homogeneous ideals of $\Z[y_1,y_2,\ldots,y_n]$: 
\begin{align*}
 &I \coloneqq (e_k(y_1,y_2,\ldots,y_n) \mid 1\le k\le n ), \\
 &I' \coloneqq ((y_i-y_{i+1})e_k(y_1,\ldots,y_i) \mid 1\le i\le n-1,\ 1\le k\le \min\{i,n-i\}), 
\end{align*}
where $e_k$ denotes the $k$-th elementary symmetric polynomial. We now state the second theorem of this paper.

\begin{theorem}\label{thm: B}
The map $\RP$ induces an isomorphism
\begin{align*}
H^*(\Pet{n};\Z)\cong  \Z[y_1,y_2,\ldots,y_n]/(I+I')
\end{align*}
as graded rings.
\end{theorem}

Explicit presentations of the cohomology ring $H^*(\Pet{n};\C)$ were given in \cite{fu-ha-ma, ha-ho-ma, HaTy} as mentioned above, and the definition of the ideal $I'$ is motivated algebraically by \cite{fu-ha-ma, ha-ho-ma} and geometrically by \cite{AHKZ}. See Section~\ref{sec: ring presentation} for details.

\bigskip
\noindent \textbf{Acknowledgments}.
We are grateful to Mikiya Masuda, Hideya Kuwata, and Tatsuya Horiguchi for valuable  discussions. This research is supported in part by Osaka City University Advanced Mathematical Institute (MEXT Joint Usage/Research Center on Mathematics and Theoretical Physics): Topology and combinatorics of Hessenberg varieties. The first author is supported in part by JSPS Grant-in-Aid for Early-Career Scientists: 18K13413. The second author is supported in part by NSFC: 11901218.

\bigskip
\vspace{10pt}

\section{Background and preliminaries}
\subsection{Hessenberg varieties}
Let $n(\ge2)$ be a positive integer. We use the notation $[n]\coloneqq\{1,2,\ldots,n\}$ throughout this paper. 
A function $h\colon [n]\rightarrow [n]$ is a \textit{Hessenberg function} if it satisfies the following two conditions:
\begin{itemize}
 \item[(i)] $h(1)\le h(2)\le \cdots \le h(n)$, 
 \item[(ii)] $h(j)\ge j$ for all $j\in[n]$.
\end{itemize}
We identify a Hessenberg function $h$ with a configuration of shaded boxes on a square grid of size $n\times n$ which consists of boxes in the $i$-th row and the $j$-th column satisfying $i\le h(j)$ for $i,j\in [n]$.

\begin{example}\label{ex: hess function}
{\rm 
If $n=5$ and $h\colon [5]\rightarrow[5]$ is given by
\begin{align*}
 (h(1),h(2),h(3),h(4),h(5)) = (3,3,4,5,5),
\end{align*}
then $h$ is a Hessenberg function corresponding to the configuration of the shaded boxes drawn in Figure~\ref{pic: example of hess function}.}
\end{example}

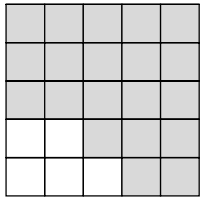
\begin{figure}[htbp]
{\unitlength 0.1in%
\begin{picture}(10.0000,10.0000)(22.0000,-18.0000)%
%
\special{pn 0}%
\special{sh 0.150}%
\special{pa 2200 800}%
\special{pa 2400 800}%
\special{pa 2400 1000}%
\special{pa 2200 1000}%
\special{pa 2200 800}%
\special{ip}%
\special{pn 8}%
\special{pa 2200 800}%
\special{pa 2400 800}%
\special{pa 2400 1000}%
\special{pa 2200 1000}%
\special{pa 2200 800}%
\special{pa 2400 800}%
\special{fp}%
%
\special{pn 0}%
\special{sh 0.150}%
\special{pa 2400 800}%
\special{pa 2600 800}%
\special{pa 2600 1000}%
\special{pa 2400 1000}%
\special{pa 2400 800}%
\special{ip}%
\special{pn 8}%
\special{pa 2400 800}%
\special{pa 2600 800}%
\special{pa 2600 1000}%
\special{pa 2400 1000}%
\special{pa 2400 800}%
\special{pa 2600 800}%
\special{fp}%
%
\special{pn 0}%
\special{sh 0.150}%
\special{pa 2600 800}%
\special{pa 2800 800}%
\special{pa 2800 1000}%
\special{pa 2600 1000}%
\special{pa 2600 800}%
\special{ip}%
\special{pn 8}%
\special{pa 2600 800}%
\special{pa 2800 800}%
\special{pa 2800 1000}%
\special{pa 2600 1000}%
\special{pa 2600 800}%
\special{pa 2800 800}%
\special{fp}%
%
\special{pn 0}%
\special{sh 0.150}%
\special{pa 2800 800}%
\special{pa 3000 800}%
\special{pa 3000 1000}%
\special{pa 2800 1000}%
\special{pa 2800 800}%
\special{ip}%
\special{pn 8}%
\special{pa 2800 800}%
\special{pa 3000 800}%
\special{pa 3000 1000}%
\special{pa 2800 1000}%
\special{pa 2800 800}%
\special{pa 3000 800}%
\special{fp}%
%
\special{pn 0}%
\special{sh 0.150}%
\special{pa 3000 800}%
\special{pa 3200 800}%
\special{pa 3200 1000}%
\special{pa 3000 1000}%
\special{pa 3000 800}%
\special{ip}%
\special{pn 8}%
\special{pa 3000 800}%
\special{pa 3200 800}%
\special{pa 3200 1000}%
\special{pa 3000 1000}%
\special{pa 3000 800}%
\special{pa 3200 800}%
\special{fp}%
%
\special{pn 0}%
\special{sh 0.150}%
\special{pa 2200 1000}%
\special{pa 2400 1000}%
\special{pa 2400 1200}%
\special{pa 2200 1200}%
\special{pa 2200 1000}%
\special{ip}%
\special{pn 8}%
\special{pa 2200 1000}%
\special{pa 2400 1000}%
\special{pa 2400 1200}%
\special{pa 2200 1200}%
\special{pa 2200 1000}%
\special{pa 2400 1000}%
\special{fp}%
%
\special{pn 0}%
\special{sh 0.150}%
\special{pa 2400 1000}%
\special{pa 2600 1000}%
\special{pa 2600 1200}%
\special{pa 2400 1200}%
\special{pa 2400 1000}%
\special{ip}%
\special{pn 8}%
\special{pa 2400 1000}%
\special{pa 2600 1000}%
\special{pa 2600 1200}%
\special{pa 2400 1200}%
\special{pa 2400 1000}%
\special{pa 2600 1000}%
\special{fp}%
%
\special{pn 0}%
\special{sh 0.150}%
\special{pa 2600 1000}%
\special{pa 2800 1000}%
\special{pa 2800 1200}%
\special{pa 2600 1200}%
\special{pa 2600 1000}%
\special{ip}%
\special{pn 8}%
\special{pa 2600 1000}%
\special{pa 2800 1000}%
\special{pa 2800 1200}%
\special{pa 2600 1200}%
\special{pa 2600 1000}%
\special{pa 2800 1000}%
\special{fp}%
%
\special{pn 0}%
\special{sh 0.150}%
\special{pa 2800 1000}%
\special{pa 3000 1000}%
\special{pa 3000 1200}%
\special{pa 2800 1200}%
\special{pa 2800 1000}%
\special{ip}%
\special{pn 8}%
\special{pa 2800 1000}%
\special{pa 3000 1000}%
\special{pa 3000 1200}%
\special{pa 2800 1200}%
\special{pa 2800 1000}%
\special{pa 3000 1000}%
\special{fp}%
%
\special{pn 0}%
\special{sh 0.150}%
\special{pa 3000 1000}%
\special{pa 3200 1000}%
\special{pa 3200 1200}%
\special{pa 3000 1200}%
\special{pa 3000 1000}%
\special{ip}%
\special{pn 8}%
\special{pa 3000 1000}%
\special{pa 3200 1000}%
\special{pa 3200 1200}%
\special{pa 3000 1200}%
\special{pa 3000 1000}%
\special{pa 3200 1000}%
\special{fp}%
%
\special{pn 0}%
\special{sh 0.150}%
\special{pa 2200 1200}%
\special{pa 2400 1200}%
\special{pa 2400 1400}%
\special{pa 2200 1400}%
\special{pa 2200 1200}%
\special{ip}%
\special{pn 8}%
\special{pa 2200 1200}%
\special{pa 2400 1200}%
\special{pa 2400 1400}%
\special{pa 2200 1400}%
\special{pa 2200 1200}%
\special{pa 2400 1200}%
\special{fp}%
%
\special{pn 0}%
\special{sh 0.150}%
\special{pa 2400 1200}%
\special{pa 2600 1200}%
\special{pa 2600 1400}%
\special{pa 2400 1400}%
\special{pa 2400 1200}%
\special{ip}%
\special{pn 8}%
\special{pa 2400 1200}%
\special{pa 2600 1200}%
\special{pa 2600 1400}%
\special{pa 2400 1400}%
\special{pa 2400 1200}%
\special{pa 2600 1200}%
\special{fp}%
%
\special{pn 0}%
\special{sh 0.150}%
\special{pa 2600 1200}%
\special{pa 2800 1200}%
\special{pa 2800 1400}%
\special{pa 2600 1400}%
\special{pa 2600 1200}%
\special{ip}%
\special{pn 8}%
\special{pa 2600 1200}%
\special{pa 2800 1200}%
\special{pa 2800 1400}%
\special{pa 2600 1400}%
\special{pa 2600 1200}%
\special{pa 2800 1200}%
\special{fp}%
%
\special{pn 0}%
\special{sh 0.150}%
\special{pa 2800 1200}%
\special{pa 3000 1200}%
\special{pa 3000 1400}%
\special{pa 2800 1400}%
\special{pa 2800 1200}%
\special{ip}%
\special{pn 8}%
\special{pa 2800 1200}%
\special{pa 3000 1200}%
\special{pa 3000 1400}%
\special{pa 2800 1400}%
\special{pa 2800 1200}%
\special{pa 3000 1200}%
\special{fp}%
%
\special{pn 0}%
\special{sh 0.150}%
\special{pa 3000 1200}%
\special{pa 3200 1200}%
\special{pa 3200 1400}%
\special{pa 3000 1400}%
\special{pa 3000 1200}%
\special{ip}%
\special{pn 8}%
\special{pa 3000 1200}%
\special{pa 3200 1200}%
\special{pa 3200 1400}%
\special{pa 3000 1400}%
\special{pa 3000 1200}%
\special{pa 3200 1200}%
\special{fp}%
%
\special{pn 8}%
\special{pa 2200 1400}%
\special{pa 2400 1400}%
\special{pa 2400 1600}%
\special{pa 2200 1600}%
\special{pa 2200 1400}%
\special{pa 2400 1400}%
\special{fp}%
%
\special{pn 8}%
\special{pa 2400 1400}%
\special{pa 2600 1400}%
\special{pa 2600 1600}%
\special{pa 2400 1600}%
\special{pa 2400 1400}%
\special{pa 2600 1400}%
\special{fp}%
%
\special{pn 0}%
\special{sh 0.150}%
\special{pa 2600 1400}%
\special{pa 2800 1400}%
\special{pa 2800 1600}%
\special{pa 2600 1600}%
\special{pa 2600 1400}%
\special{ip}%
\special{pn 8}%
\special{pa 2600 1400}%
\special{pa 2800 1400}%
\special{pa 2800 1600}%
\special{pa 2600 1600}%
\special{pa 2600 1400}%
\special{pa 2800 1400}%
\special{fp}%
%
\special{pn 0}%
\special{sh 0.150}%
\special{pa 2800 1400}%
\special{pa 3000 1400}%
\special{pa 3000 1600}%
\special{pa 2800 1600}%
\special{pa 2800 1400}%
\special{ip}%
\special{pn 8}%
\special{pa 2800 1400}%
\special{pa 3000 1400}%
\special{pa 3000 1600}%
\special{pa 2800 1600}%
\special{pa 2800 1400}%
\special{pa 3000 1400}%
\special{fp}%
%
\special{pn 0}%
\special{sh 0.150}%
\special{pa 3000 1400}%
\special{pa 3200 1400}%
\special{pa 3200 1600}%
\special{pa 3000 1600}%
\special{pa 3000 1400}%
\special{ip}%
\special{pn 8}%
\special{pa 3000 1400}%
\special{pa 3200 1400}%
\special{pa 3200 1600}%
\special{pa 3000 1600}%
\special{pa 3000 1400}%
\special{pa 3200 1400}%
\special{fp}%
%
\special{pn 8}%
\special{pa 2200 1600}%
\special{pa 2400 1600}%
\special{pa 2400 1800}%
\special{pa 2200 1800}%
\special{pa 2200 1600}%
\special{pa 2400 1600}%
\special{fp}%
%
\special{pn 8}%
\special{pa 2400 1600}%
\special{pa 2600 1600}%
\special{pa 2600 1800}%
\special{pa 2400 1800}%
\special{pa 2400 1600}%
\special{pa 2600 1600}%
\special{fp}%
%
\special{pn 8}%
\special{pa 2600 1600}%
\special{pa 2800 1600}%
\special{pa 2800 1800}%
\special{pa 2600 1800}%
\special{pa 2600 1600}%
\special{pa 2800 1600}%
\special{fp}%
%
\special{pn 0}%
\special{sh 0.150}%
\special{pa 2800 1600}%
\special{pa 3000 1600}%
\special{pa 3000 1800}%
\special{pa 2800 1800}%
\special{pa 2800 1600}%
\special{ip}%
\special{pn 8}%
\special{pa 2800 1600}%
\special{pa 3000 1600}%
\special{pa 3000 1800}%
\special{pa 2800 1800}%
\special{pa 2800 1600}%
\special{pa 3000 1600}%
\special{fp}%
%
\special{pn 0}%
\special{sh 0.150}%
\special{pa 3000 1600}%
\special{pa 3200 1600}%
\special{pa 3200 1800}%
\special{pa 3000 1800}%
\special{pa 3000 1600}%
\special{ip}%
\special{pn 8}%
\special{pa 3000 1600}%
\special{pa 3200 1600}%
\special{pa 3200 1800}%
\special{pa 3000 1800}%
\special{pa 3000 1600}%
\special{pa 3200 1600}%
\special{fp}%
\end{picture}}%
\caption{The configuration of shaded boxes for Example \ref{ex: hess function}}
\label{pic: example of hess function}
\end{figure}

Let $Fl_n=Fl(\C^n)$ be the flag variety of $\C^n$. 
For an $n\times n$ matrix $X$ (viewed as a linear map $X\colon \C^n\rightarrow \C^n$) and a Hessenberg function $h\colon [n]\rightarrow [n]$, the \textit{Hessenberg variety} associated with $X$ and $h$ is defined by 
\begin{align*}
 \Hess(X,h) \coloneqq \{ V_{\bullet}\in Fl_n \mid XV_j \subseteq V_{h(j)} \text{ for all $1\le j\le n$}\}.
\end{align*}
The Peterson variety and the permutohedral variety are both special cases of Hessenberg varieties as we will see in the next subsection.

We denote by $\text{GL}_n(\C)$ the complex general linear group of degree $n$. There is a natural action of $\text{GL}_n(\C)$ on $Fl_n$, and we have $\Hess(gXg^{-1},h) = g \cdot \Hess(X,h)$ in $Fl_n$.
This implies that 
\begin{align}\label{eq: similarity}
 \Hess(gXg^{-1},h) \cong \Hess(X,h)
\end{align}
for $g\in \text{GL}_n(\C)$ 
so that taking conjugation of the matrix $X$ does not change the isomorphism class of $\Hess(X,h)$.

\subsection{Peterson varieties}\label{subsect: Peterson varieties}
Let $N$ be an $n\times n$ regular nilpotent matrix (i.e., a nilpotent matrix consisting of a single Jordan block), and let $h_2\colon [n]\rightarrow[n]$ be the Hessenberg function given by
\begin{align}\label{eq: toric Hessenberg function}
 h_2(j)=j+1 \quad \text{for $1\le j<n$}.
\end{align}
The \textit{Peterson variety} $\Pet{n}$ is defined as a special case of Hessenberg varieties: 
\begin{align}\label{eq: def of Pet}
 \Pet{n} \coloneqq \Hess(N,h_2) = \{ V_{\bullet}\in Fl_n \mid NV_j \subseteq V_{j+1} \text{ for all $1\le j<n$}\}.
\end{align}
For simplicity, we assume that  $N$ is in Jordan canonical form in the rest of this paper.
It is well-known (cf.\ \cite{Peterson} or \cite[Lemma 7.1]{an-ty}) that 
\begin{align}\label{eq: dim of Pet}
 \dim_{\C}\Pet{n} = n-1.
\end{align}
For a topological space $X$, we denote by $H_*(X;\Z)$ and $H^*(X;\Z)$ the singular homology group of $X$ and the singular cohomology ring of $X$, respectively.
We set $H_{\text{odd}}(X;\Z)\coloneqq \oplus_{k\ge0} H_{2k+1}(X;\Z)$.

\begin{proposition}\label{prop: known things 0 Pet}
$($\cite{Peterson} and \cite[Theorem 7.1]{ty}$)$
\begin{itemize}
 \item[(i)] $H_*(\Pet{n};\Z)$ is a torsion free $\Z$-module of rank $2^{n-1}$.
 \item[(ii)] $H_{\text{\rm odd}}(\Pet{n};\Z)=0$.
\end{itemize}
\end{proposition}

\vspace{10pt}

\subsection{Permutohedral varieties}\label{subsect: permutohedral}
Let $S$ be an $n\times n$ regular semisimple matrix (i.e., an $n\times n$ matrix with $n$ distinct eigenvalues), and let $h_2:[n]\rightarrow[n]$ be the Hessenberg function defined in \eqref{eq: toric Hessenberg function}.
The \textit{permutohedral variety} $\Perm{n}$ is also a special case of Hessenberg varieties:
\begin{align}\label{eq: def of Perm}
 \Perm{n} \coloneqq \Hess(S,h_2) = \{ V_{\bullet}\in Fl_n \mid SV_j \subseteq V_{j+1} \text{ for all $1\le j<n$}\}.
\end{align}
It is known that $\Perm{n}$ is the non-singular projective toric variety associated with the fan consisting of the set of Weyl chambers of type A$_{n-1}$ (\cite[Theorems 6 and 11]{De Mari-Procesi-Shayman}). This implies that the isomorphism class of $\Hess(S,h_2)$ does not depend on a choice of a regular semisimple matrix $S$.
It also follows from \cite[Theorem 11]{De Mari-Procesi-Shayman} that 
\begin{align*}
 \dim_{\C}\Perm{n} = n-1.
\end{align*}

\begin{proposition}\label{prop: known things 0 Perm}
$($\cite[Section III]{De Mari-Procesi-Shayman}$)$
\begin{itemize}
 \item[(i)] $H_*(\Perm{n};\Z)$ is a torsion free $\Z$-module of rank  $n!$.
 \item[(ii)] $H_{\text{\rm odd}}(\Perm{n};\Z)=0$.
\end{itemize}
\end{proposition}

The Weyl group $\Sn$ permutes the set of Weyl chambers of type A$_{n-1}$, and hence it induces an $\Sn$-action on the cohomology ring $H^*(\Perm{n};\Z)$ of the toric variety $\Perm{n}$. It is known $($e.g., \cite[Sect.\ 1]{ChoHongLee1}$)$ that this $\Sn$-module can also be constructed as a special case of the \textit{dot action} due to Tymoczko (\cite{Ty3}) which we briefly review in what follows.

Recalling that we have \eqref{eq: similarity}, 
we may assume that the matrix $S$ in the diagonal form.
Let $h\colon [n]\rightarrow [n]$ be an arbitrary Hessenberg function.
We denote by $T$ the maximal torus of $\text{GL}_n(\C)$ consisting of diagonal matrices.
There is a natural action of $\text{GL}_n(\C)$ on $Fl_n$, and hence $T$ acts on $Fl_n$ through the action of $\text{GL}_n(\C)$.
This $T$-action preserves $\Hess(S,h)\subseteq Fl_n$ since the matrix $S$ and elements of $T$ commute.
In this way, we obtain a $T$-action on $\Hess(S,h)$.
In \cite{Ty3}, Tymoczko constructed a representation of $\Sn$ on the $T$-equivariant cohomology $H_T^*(\Hess(S,h);\C)$ by using its GKM presentation, and she showed that it induces a representation of $\Sn$ on the ordinary cohomology ring $H^*(\Hess(S,h);\C)$ which preserves the degree and the cup product.
As mentioned in \cite[Remark 2.4]{ab-ho-ma}, the same construction works for the integral cohomology ring $H^*(\Hess(S,h);\Z)$ as well. 
Since we have $\Perm{n} =\Hess(S,h_2)$ by definition, we regard $H^*(\Perm{n};\Z)$ as an $\Sn$-module by this way throughout the paper.
For this $\Sn$-module, the following claim is deduced from \cite{AHHM,Klyachko,ste94}.

\begin{proposition}\label{prop: known things 2}
\text{}
\begin{itemize}
 \item[(i)] The image of the restriction map $H^*(Fl_n;\Z) \rightarrow  H^*(\Perm{n};\Z)$ lies in the invariant submodule $H^*(\Perm{n};\Z)^{\Sn}$.
 \item[(ii)] $\rank H^*(\Perm{n};\Z)^{\Sn} \le 2^{n-1}$.
\end{itemize}
\end{proposition}

\begin{proof}
The claim (i) for $\Q$ coefficients follows from \cite{Klyachko} or \cite[Sect.\ 8]{AHHM}. Since the argument of \cite[Sect.\ 8]{AHHM} works verbatim for $\Z$ coefficients as well, we explain only the outline of the proof.
Let $h_{n}\colon [n]\rightarrow [n]$ be the Hessenberg function given by $h_{n}(j)=n$ for $1\le j\le n$. Since $Fl_n=\Hess(S,h_{n})$, the cohomology ring $H^*(Fl_n;\Z)$ also admits Tymoczko's $\Sn$-action.
By the construction of this $\Sn$-action, it follows that the restriction map $H^*(Fl_n;\Z) \rightarrow  H^*(\Perm{n};\Z)$ is a homomorphism of $\Sn$-modules (\cite[Lemma~8.1]{AHHM}). Also, it is known from \cite[Proposition~4.4]{Ty3} that the $\Sn$-action on $H^*(Fl_n;\Z)$ is trivial. This proves the claim (i).

For the claim (ii), it is clear that we have the inclusion map
\begin{align*}
 H^*(\Perm{n};\Z)^{\Sn}
 \hookrightarrow
 H^*(\Perm{n};\Z).
\end{align*}
Since both of these are free $\Z$-modules by Proposition~\ref{prop: known things 0 Perm}, this map induces an injective linear map over $\C$: 
\begin{align}\label{eq: invariant ineq}
 H^*(\Perm{n};\Z)^{\Sn}\otimes_{\Z}\C
 \hookrightarrow
 H^*(\Perm{n};\Z)\otimes_{\Z}\C. 
\end{align}
Here, the target vector space $H^*(\Perm{n};\Z)\otimes_{\Z}\C$ has a natural structure of an $\Sn$-representation induced by that of $H^*(\Perm{n};\Z)$, and it is isomorphic to $H^*(\Perm{n};\C)$ as $\Sn$-representations by construction (cf.\ \cite[Remark 2.4]{ab-ho-ma}).
The image of the map \eqref{eq: invariant ineq} lies on the $\Sn$-invariant subspace of $H^*(\Perm{n};\Z)\otimes_{\Z}\C\cong H^*(\Perm{n};\C)$ so that we obtain an injective linear map
\begin{align*}
 H^*(\Perm{n};\Z)^{\Sn}\otimes_{\Z}\C
 \hookrightarrow
 H^*(\Perm{n};\C)^{\Sn}.
\end{align*}
Thus, it follows that
\begin{align*}
 \rank H^*(\Perm{n};\Z)^{\Sn}
 \le  \dim_{\C} H^*(\Perm{n};\C)^{\Sn}
 =2^{n-1},
\end{align*}
where the last equality follows from \cite[Theorem 3.1]{ste94}.
\end{proof}

\begin{remark}\label{rem: known things 2}
\textnormal{
For the second claim of Proposition~$\ref{prop: known things 2}$, we show that the equality $\rank H^*(\Perm{n};\Z)^{\Sn} = 2^{n-1}$ holds in the next section. See Remark~\ref{rem: on rank inequality} for details.
}
\end{remark}

\vspace{10pt}

\subsection{A connection between Peterson varieties and permutohedral varieties}
Let $\lambda_1,\lambda_2,\ldots,\lambda_n\in\C$ be distinct complex numbers.
For $t\in \C$, we consider an $n\times n$ matrix $S_t$ given by
\begin{align*}
 S_t =
 \begin{pmatrix}
  t\lambda_1\hspace{-2pt} & \hspace{-3pt}1 & & & \\
  & \hspace{-5pt}t\lambda_2\hspace{-5pt} & 1 & & \\
  & & \ddots & \ddots & \\
  & & & \hspace{-5pt}t\lambda_{n-1}\hspace{-5pt} & \hspace{-5pt}1 \\
  & & & & \hspace{-2pt}t\lambda_n
 \end{pmatrix}
 \quad (t\in\C).
\end{align*}
For a (fixed) Hessenberg function $h\colon [n]\rightarrow[n]$, this leads us to consider a family of Hessenberg varieties over the $1$-dimensional base space $\C$ such that the fiber over $t\in \C$ is $\Hess(S_t,h)$ (see \cite[Section 4]{ab-de-ga-ha} for details).
For our purpose, we take $h=h_2$, where $h_2$ is the Hessenberg function given in \eqref{eq: toric Hessenberg function}.
When $t\ne0$, the matrix $S_t$ is a regular semisimple matrix, and hence we have $\Hess(S_t,h_2)\cong \Perm{n}$ by \eqref{eq: similarity}. When $t=0$, it is clear that $\Hess(S_0,h_2)= \Pet{n}$. Thus, we obtain a degeneration from $\Perm{n}$ to $\Pet{n}$.
This family was studied in \cite{ab-de-ga-ha} to prove the following.

\begin{proposition}\label{prop: known things 4}
$($\cite[Corollary 4.3]{ab-de-ga-ha}$)$
We have
\begin{align*}
 [\Pet{n}]=[\Perm{n}] \quad \text{in $H_*(Fl_{n};\Z)$},
\end{align*}
where $[\Pet{n}]$ and $[\Perm{n}]$ are the cycles representing the subvarieties
$\Pet{n}$ and $\Perm{n}$ in $Fl_{n}$, respectively.
\end{proposition}

\vspace{10pt}

\subsection{Combinatorics on Dynkin diagrams of type A}\label{subsec: comb on J}
Recall from our notation that $[n-1]= \{1,2,\ldots,n-1\}$. We regard it as the set of vertices of the Dynkin diagram of type $A_{n-1}$. Namely, two vertices $i,j\in[n-1]$ are connected by an edge if and only if $|i-j|=1$. See Figure~\ref{pic: Dynkin diagram}.\\ \vspace{-5pt}
\begin{figure}[htbp]
\[
{\unitlength 0.1in%
\begin{picture}(18.8000,2.0000)(15.6000,-15.2000)%
%
\special{pn 8}%
\special{ar 1600 1400 40 40 0.0000000 6.2831853}%
%
\special{pn 8}%
\special{ar 1900 1400 40 40 0.0000000 6.2831853}%
%
\special{pn 8}%
\special{ar 2200 1400 40 40 0.0000000 6.2831853}%
%
\special{pn 8}%
\special{pa 2160 1400}%
\special{pa 1940 1400}%
\special{fp}%
%
\special{pn 8}%
\special{pa 1860 1400}%
\special{pa 1640 1400}%
\special{fp}%
%
\special{pn 8}%
\special{ar 3400 1400 40 40 0.0000000 6.2831853}%
%
\special{pn 8}%
\special{pa 3360 1400}%
\special{pa 3140 1400}%
\special{fp}%
%
\special{pn 8}%
\special{pa 2460 1400}%
\special{pa 2240 1400}%
\special{fp}%
\put(26.6000,-14.5000){\makebox(0,0)[lb]{$\cdots$}}%
\put(15.7000,-16.5000){\makebox(0,0)[lb]{$1$}}%
\put(21.7000,-16.5000){\makebox(0,0)[lb]{$3$}}%
\put(18.7000,-16.5000){\makebox(0,0)[lb]{$2$}}%
\put(32.6000,-16.5000){\makebox(0,0)[lb]{$n-1$}}%
\end{picture}}%
\]
\caption{The Dynkin diagram of type A$_{n-1}$.}
\label{pic: Dynkin diagram}
\end{figure}

We also regard each subset $J\subseteq[n-1]$ as a full-subgraph of the Dynkin diagram. We may decompose it into the connected components:
\begin{align*}
 J = J_1 \sqcup J_2 \sqcup \cdots \sqcup J_{m},
\end{align*}
where $J_k\ (1\le k\le m)$ is the set of vertices of a maximal connected subgraph of $J$. To determine each $J_k$ uniquely, we require that elements of $J_k$ are less than elements of $J_{k'}$ when $k< k'$.

\begin{example}
Let $n=10$ and $J=\{1,2,4,5,6,9\}$. Then we have 
\begin{align*}
 J_1=\{1,2\},\ J_2=\{4,5,6\},\ J_3=\{9\}
\end{align*}
so that $J = J_1\sqcup J_2\sqcup J_3 = \{1,2\} \sqcup\{4,5,6\} \sqcup\{9\}$.
\end{example}

\vspace{10pt}

For $J\subseteq[n-1]$, let us consider the associated Young subgroup 
\begin{align*}
 \mathfrak{S}_J \coloneqq 
 \mathfrak{S}_{J_1}\times \mathfrak{S}_{J_2} \times \cdots \times \mathfrak{S}_{J_m}
 \subseteq \Sn,
\end{align*}
where $\mathfrak{S}_{J_k}\ (1\le k\le m)$ is the subgroup of $\Sn$ generated by the simple reflections $s_i$ for all $i\in J_k$. 
Let $w_J$ be the longest element of $\mathfrak{S}_J$, i.e.,
\begin{align}\label{eq: def of wJ}
 w_J \coloneqq w_0^{(J_1)} w_0^{(J_2)} \cdots w_0^{(J_{m})} \in \mathfrak{S}_J, 
\end{align}
where $w_0^{(J_k)}$ is the longest element of $\mathfrak{S}_{J_k}$ $(1\le k\le m)$.

\begin{example}\label{eg: wJ}
{\rm
If $n=10$ and $J=\{1,2\} \sqcup\{4,5,6\} \sqcup\{9\} = J_1\sqcup J_2\sqcup J_3$ as above, then the permutation $w_J$ in the form of its permutation matrix is given by
\begin{align*}
 w_J 
 = w_0^{(J_1)} w_0^{(J_2)} w_0^{(J_3)}
 =
 \left(
 \begin{array}{@{\,}ccc|cccc|c|cc@{\,}}
     & & 1 & & & & & & & \\
     & 1 & & & & & & & & \\ 
    1 & & & & & & & & & \\ \hline
     & & & & & & 1& & & \\
     & & & & & 1 & & & & \\
     & & & & 1 & & & & & \\
     & & & 1 & & & & & & \\ \hline
     & & & & & & & 1 & & \\ \hline
     & & & & & & & & & 1 \\
     & & & & & & & & 1 & 
 \end{array}
 \right).
\end{align*}
}
\end{example}

\vspace{10pt}

For $J\subseteq[n-1]$, there is a natural Hessenberg function which is determined by $J$ as follows. Let $h_J\colon [n]\rightarrow[n]$ be a function given by
\begin{align}\label{eq: def of hJ}
 h_J (j) =
 \begin{cases}
  j+1 \quad &\text{if $j\in J$},\\
  j &\text{if $j\notin J$}.
 \end{cases}
\end{align}
We note that
\begin{align}\label{eq: ineq for hJ}
 h_J(j)\le h_2(j) \quad \text{for $1\le j\le n$},
\end{align}
where $h_2$ is the Hessenberg function defined in \eqref{eq: toric Hessenberg function}.

\begin{example}
{\rm
If $n=10$ and $J=\{1,2,4,5,6,9\}$ as above, then the configuration of boxes of $h_J$ is given in Figure $\ref{pic: example of hJ}$ (cf.\ Example~\ref{eg: wJ}).}
\begin{figure}[htbp]
{\unitlength 0.1in%
\begin{picture}(20.0000,20.0000)(36.0000,-28.0000)%
%
\special{pn 0}%
\special{sh 0.150}%
\special{pa 3600 800}%
\special{pa 3800 800}%
\special{pa 3800 1000}%
\special{pa 3600 1000}%
\special{pa 3600 800}%
\special{ip}%
\special{pn 8}%
\special{pa 3600 800}%
\special{pa 3800 800}%
\special{pa 3800 1000}%
\special{pa 3600 1000}%
\special{pa 3600 800}%
\special{pa 3800 800}%
\special{fp}%
%
\special{pn 0}%
\special{sh 0.150}%
\special{pa 3800 800}%
\special{pa 4000 800}%
\special{pa 4000 1000}%
\special{pa 3800 1000}%
\special{pa 3800 800}%
\special{ip}%
\special{pn 8}%
\special{pa 3800 800}%
\special{pa 4000 800}%
\special{pa 4000 1000}%
\special{pa 3800 1000}%
\special{pa 3800 800}%
\special{pa 4000 800}%
\special{fp}%
%
\special{pn 0}%
\special{sh 0.150}%
\special{pa 4000 800}%
\special{pa 4200 800}%
\special{pa 4200 1000}%
\special{pa 4000 1000}%
\special{pa 4000 800}%
\special{ip}%
\special{pn 8}%
\special{pa 4000 800}%
\special{pa 4200 800}%
\special{pa 4200 1000}%
\special{pa 4000 1000}%
\special{pa 4000 800}%
\special{pa 4200 800}%
\special{fp}%
%
\special{pn 0}%
\special{sh 0.150}%
\special{pa 4200 800}%
\special{pa 4400 800}%
\special{pa 4400 1000}%
\special{pa 4200 1000}%
\special{pa 4200 800}%
\special{ip}%
\special{pn 8}%
\special{pa 4200 800}%
\special{pa 4400 800}%
\special{pa 4400 1000}%
\special{pa 4200 1000}%
\special{pa 4200 800}%
\special{pa 4400 800}%
\special{fp}%
%
\special{pn 0}%
\special{sh 0.150}%
\special{pa 4400 800}%
\special{pa 4600 800}%
\special{pa 4600 1000}%
\special{pa 4400 1000}%
\special{pa 4400 800}%
\special{ip}%
\special{pn 8}%
\special{pa 4400 800}%
\special{pa 4600 800}%
\special{pa 4600 1000}%
\special{pa 4400 1000}%
\special{pa 4400 800}%
\special{pa 4600 800}%
\special{fp}%
%
\special{pn 0}%
\special{sh 0.150}%
\special{pa 4600 800}%
\special{pa 4800 800}%
\special{pa 4800 1000}%
\special{pa 4600 1000}%
\special{pa 4600 800}%
\special{ip}%
\special{pn 8}%
\special{pa 4600 800}%
\special{pa 4800 800}%
\special{pa 4800 1000}%
\special{pa 4600 1000}%
\special{pa 4600 800}%
\special{pa 4800 800}%
\special{fp}%
%
\special{pn 0}%
\special{sh 0.150}%
\special{pa 4800 800}%
\special{pa 5000 800}%
\special{pa 5000 1000}%
\special{pa 4800 1000}%
\special{pa 4800 800}%
\special{ip}%
\special{pn 8}%
\special{pa 4800 800}%
\special{pa 5000 800}%
\special{pa 5000 1000}%
\special{pa 4800 1000}%
\special{pa 4800 800}%
\special{pa 5000 800}%
\special{fp}%
%
\special{pn 0}%
\special{sh 0.150}%
\special{pa 5000 800}%
\special{pa 5200 800}%
\special{pa 5200 1000}%
\special{pa 5000 1000}%
\special{pa 5000 800}%
\special{ip}%
\special{pn 8}%
\special{pa 5000 800}%
\special{pa 5200 800}%
\special{pa 5200 1000}%
\special{pa 5000 1000}%
\special{pa 5000 800}%
\special{pa 5200 800}%
\special{fp}%
%
\special{pn 0}%
\special{sh 0.150}%
\special{pa 5200 800}%
\special{pa 5400 800}%
\special{pa 5400 1000}%
\special{pa 5200 1000}%
\special{pa 5200 800}%
\special{ip}%
\special{pn 8}%
\special{pa 5200 800}%
\special{pa 5400 800}%
\special{pa 5400 1000}%
\special{pa 5200 1000}%
\special{pa 5200 800}%
\special{pa 5400 800}%
\special{fp}%
%
\special{pn 0}%
\special{sh 0.150}%
\special{pa 5400 800}%
\special{pa 5600 800}%
\special{pa 5600 1000}%
\special{pa 5400 1000}%
\special{pa 5400 800}%
\special{ip}%
\special{pn 8}%
\special{pa 5400 800}%
\special{pa 5600 800}%
\special{pa 5600 1000}%
\special{pa 5400 1000}%
\special{pa 5400 800}%
\special{pa 5600 800}%
\special{fp}%
%
\special{pn 0}%
\special{sh 0.150}%
\special{pa 3600 1000}%
\special{pa 3800 1000}%
\special{pa 3800 1200}%
\special{pa 3600 1200}%
\special{pa 3600 1000}%
\special{ip}%
\special{pn 8}%
\special{pa 3600 1000}%
\special{pa 3800 1000}%
\special{pa 3800 1200}%
\special{pa 3600 1200}%
\special{pa 3600 1000}%
\special{pa 3800 1000}%
\special{fp}%
%
\special{pn 0}%
\special{sh 0.150}%
\special{pa 3800 1000}%
\special{pa 4000 1000}%
\special{pa 4000 1200}%
\special{pa 3800 1200}%
\special{pa 3800 1000}%
\special{ip}%
\special{pn 8}%
\special{pa 3800 1000}%
\special{pa 4000 1000}%
\special{pa 4000 1200}%
\special{pa 3800 1200}%
\special{pa 3800 1000}%
\special{pa 4000 1000}%
\special{fp}%
%
\special{pn 0}%
\special{sh 0.150}%
\special{pa 4000 1000}%
\special{pa 4200 1000}%
\special{pa 4200 1200}%
\special{pa 4000 1200}%
\special{pa 4000 1000}%
\special{ip}%
\special{pn 8}%
\special{pa 4000 1000}%
\special{pa 4200 1000}%
\special{pa 4200 1200}%
\special{pa 4000 1200}%
\special{pa 4000 1000}%
\special{pa 4200 1000}%
\special{fp}%
%
\special{pn 0}%
\special{sh 0.150}%
\special{pa 4200 1000}%
\special{pa 4400 1000}%
\special{pa 4400 1200}%
\special{pa 4200 1200}%
\special{pa 4200 1000}%
\special{ip}%
\special{pn 8}%
\special{pa 4200 1000}%
\special{pa 4400 1000}%
\special{pa 4400 1200}%
\special{pa 4200 1200}%
\special{pa 4200 1000}%
\special{pa 4400 1000}%
\special{fp}%
%
\special{pn 0}%
\special{sh 0.150}%
\special{pa 4400 1000}%
\special{pa 4600 1000}%
\special{pa 4600 1200}%
\special{pa 4400 1200}%
\special{pa 4400 1000}%
\special{ip}%
\special{pn 8}%
\special{pa 4400 1000}%
\special{pa 4600 1000}%
\special{pa 4600 1200}%
\special{pa 4400 1200}%
\special{pa 4400 1000}%
\special{pa 4600 1000}%
\special{fp}%
%
\special{pn 0}%
\special{sh 0.150}%
\special{pa 4600 1000}%
\special{pa 4800 1000}%
\special{pa 4800 1200}%
\special{pa 4600 1200}%
\special{pa 4600 1000}%
\special{ip}%
\special{pn 8}%
\special{pa 4600 1000}%
\special{pa 4800 1000}%
\special{pa 4800 1200}%
\special{pa 4600 1200}%
\special{pa 4600 1000}%
\special{pa 4800 1000}%
\special{fp}%
%
\special{pn 0}%
\special{sh 0.150}%
\special{pa 4800 1000}%
\special{pa 5000 1000}%
\special{pa 5000 1200}%
\special{pa 4800 1200}%
\special{pa 4800 1000}%
\special{ip}%
\special{pn 8}%
\special{pa 4800 1000}%
\special{pa 5000 1000}%
\special{pa 5000 1200}%
\special{pa 4800 1200}%
\special{pa 4800 1000}%
\special{pa 5000 1000}%
\special{fp}%
%
\special{pn 0}%
\special{sh 0.150}%
\special{pa 5000 1000}%
\special{pa 5200 1000}%
\special{pa 5200 1200}%
\special{pa 5000 1200}%
\special{pa 5000 1000}%
\special{ip}%
\special{pn 8}%
\special{pa 5000 1000}%
\special{pa 5200 1000}%
\special{pa 5200 1200}%
\special{pa 5000 1200}%
\special{pa 5000 1000}%
\special{pa 5200 1000}%
\special{fp}%
%
\special{pn 0}%
\special{sh 0.150}%
\special{pa 5200 1000}%
\special{pa 5400 1000}%
\special{pa 5400 1200}%
\special{pa 5200 1200}%
\special{pa 5200 1000}%
\special{ip}%
\special{pn 8}%
\special{pa 5200 1000}%
\special{pa 5400 1000}%
\special{pa 5400 1200}%
\special{pa 5200 1200}%
\special{pa 5200 1000}%
\special{pa 5400 1000}%
\special{fp}%
%
\special{pn 0}%
\special{sh 0.150}%
\special{pa 5400 1000}%
\special{pa 5600 1000}%
\special{pa 5600 1200}%
\special{pa 5400 1200}%
\special{pa 5400 1000}%
\special{ip}%
\special{pn 8}%
\special{pa 5400 1000}%
\special{pa 5600 1000}%
\special{pa 5600 1200}%
\special{pa 5400 1200}%
\special{pa 5400 1000}%
\special{pa 5600 1000}%
\special{fp}%
%
\special{pn 8}%
\special{pa 3600 1200}%
\special{pa 3800 1200}%
\special{pa 3800 1400}%
\special{pa 3600 1400}%
\special{pa 3600 1200}%
\special{pa 3800 1200}%
\special{fp}%
%
\special{pn 0}%
\special{sh 0.150}%
\special{pa 3800 1200}%
\special{pa 4000 1200}%
\special{pa 4000 1400}%
\special{pa 3800 1400}%
\special{pa 3800 1200}%
\special{ip}%
\special{pn 8}%
\special{pa 3800 1200}%
\special{pa 4000 1200}%
\special{pa 4000 1400}%
\special{pa 3800 1400}%
\special{pa 3800 1200}%
\special{pa 4000 1200}%
\special{fp}%
%
\special{pn 0}%
\special{sh 0.150}%
\special{pa 4000 1200}%
\special{pa 4200 1200}%
\special{pa 4200 1400}%
\special{pa 4000 1400}%
\special{pa 4000 1200}%
\special{ip}%
\special{pn 8}%
\special{pa 4000 1200}%
\special{pa 4200 1200}%
\special{pa 4200 1400}%
\special{pa 4000 1400}%
\special{pa 4000 1200}%
\special{pa 4200 1200}%
\special{fp}%
%
\special{pn 0}%
\special{sh 0.150}%
\special{pa 4200 1200}%
\special{pa 4400 1200}%
\special{pa 4400 1400}%
\special{pa 4200 1400}%
\special{pa 4200 1200}%
\special{ip}%
\special{pn 8}%
\special{pa 4200 1200}%
\special{pa 4400 1200}%
\special{pa 4400 1400}%
\special{pa 4200 1400}%
\special{pa 4200 1200}%
\special{pa 4400 1200}%
\special{fp}%
%
\special{pn 0}%
\special{sh 0.150}%
\special{pa 4400 1200}%
\special{pa 4600 1200}%
\special{pa 4600 1400}%
\special{pa 4400 1400}%
\special{pa 4400 1200}%
\special{ip}%
\special{pn 8}%
\special{pa 4400 1200}%
\special{pa 4600 1200}%
\special{pa 4600 1400}%
\special{pa 4400 1400}%
\special{pa 4400 1200}%
\special{pa 4600 1200}%
\special{fp}%
%
\special{pn 0}%
\special{sh 0.150}%
\special{pa 4600 1200}%
\special{pa 4800 1200}%
\special{pa 4800 1400}%
\special{pa 4600 1400}%
\special{pa 4600 1200}%
\special{ip}%
\special{pn 8}%
\special{pa 4600 1200}%
\special{pa 4800 1200}%
\special{pa 4800 1400}%
\special{pa 4600 1400}%
\special{pa 4600 1200}%
\special{pa 4800 1200}%
\special{fp}%
%
\special{pn 0}%
\special{sh 0.150}%
\special{pa 4800 1200}%
\special{pa 5000 1200}%
\special{pa 5000 1400}%
\special{pa 4800 1400}%
\special{pa 4800 1200}%
\special{ip}%
\special{pn 8}%
\special{pa 4800 1200}%
\special{pa 5000 1200}%
\special{pa 5000 1400}%
\special{pa 4800 1400}%
\special{pa 4800 1200}%
\special{pa 5000 1200}%
\special{fp}%
%
\special{pn 0}%
\special{sh 0.150}%
\special{pa 5000 1200}%
\special{pa 5200 1200}%
\special{pa 5200 1400}%
\special{pa 5000 1400}%
\special{pa 5000 1200}%
\special{ip}%
\special{pn 8}%
\special{pa 5000 1200}%
\special{pa 5200 1200}%
\special{pa 5200 1400}%
\special{pa 5000 1400}%
\special{pa 5000 1200}%
\special{pa 5200 1200}%
\special{fp}%
%
\special{pn 0}%
\special{sh 0.150}%
\special{pa 5200 1200}%
\special{pa 5400 1200}%
\special{pa 5400 1400}%
\special{pa 5200 1400}%
\special{pa 5200 1200}%
\special{ip}%
\special{pn 8}%
\special{pa 5200 1200}%
\special{pa 5400 1200}%
\special{pa 5400 1400}%
\special{pa 5200 1400}%
\special{pa 5200 1200}%
\special{pa 5400 1200}%
\special{fp}%
%
\special{pn 0}%
\special{sh 0.150}%
\special{pa 5400 1200}%
\special{pa 5600 1200}%
\special{pa 5600 1400}%
\special{pa 5400 1400}%
\special{pa 5400 1200}%
\special{ip}%
\special{pn 8}%
\special{pa 5400 1200}%
\special{pa 5600 1200}%
\special{pa 5600 1400}%
\special{pa 5400 1400}%
\special{pa 5400 1200}%
\special{pa 5600 1200}%
\special{fp}%
%
\special{pn 8}%
\special{pa 3600 1400}%
\special{pa 3800 1400}%
\special{pa 3800 1600}%
\special{pa 3600 1600}%
\special{pa 3600 1400}%
\special{pa 3800 1400}%
\special{fp}%
%
\special{pn 8}%
\special{pa 3800 1400}%
\special{pa 4000 1400}%
\special{pa 4000 1600}%
\special{pa 3800 1600}%
\special{pa 3800 1400}%
\special{pa 4000 1400}%
\special{fp}%
%
\special{pn 8}%
\special{pa 4000 1400}%
\special{pa 4200 1400}%
\special{pa 4200 1600}%
\special{pa 4000 1600}%
\special{pa 4000 1400}%
\special{pa 4200 1400}%
\special{fp}%
%
\special{pn 0}%
\special{sh 0.150}%
\special{pa 4200 1400}%
\special{pa 4400 1400}%
\special{pa 4400 1600}%
\special{pa 4200 1600}%
\special{pa 4200 1400}%
\special{ip}%
\special{pn 8}%
\special{pa 4200 1400}%
\special{pa 4400 1400}%
\special{pa 4400 1600}%
\special{pa 4200 1600}%
\special{pa 4200 1400}%
\special{pa 4400 1400}%
\special{fp}%
%
\special{pn 0}%
\special{sh 0.150}%
\special{pa 4400 1400}%
\special{pa 4600 1400}%
\special{pa 4600 1600}%
\special{pa 4400 1600}%
\special{pa 4400 1400}%
\special{ip}%
\special{pn 8}%
\special{pa 4400 1400}%
\special{pa 4600 1400}%
\special{pa 4600 1600}%
\special{pa 4400 1600}%
\special{pa 4400 1400}%
\special{pa 4600 1400}%
\special{fp}%
%
\special{pn 0}%
\special{sh 0.150}%
\special{pa 4600 1400}%
\special{pa 4800 1400}%
\special{pa 4800 1600}%
\special{pa 4600 1600}%
\special{pa 4600 1400}%
\special{ip}%
\special{pn 8}%
\special{pa 4600 1400}%
\special{pa 4800 1400}%
\special{pa 4800 1600}%
\special{pa 4600 1600}%
\special{pa 4600 1400}%
\special{pa 4800 1400}%
\special{fp}%
%
\special{pn 0}%
\special{sh 0.150}%
\special{pa 4800 1400}%
\special{pa 5000 1400}%
\special{pa 5000 1600}%
\special{pa 4800 1600}%
\special{pa 4800 1400}%
\special{ip}%
\special{pn 8}%
\special{pa 4800 1400}%
\special{pa 5000 1400}%
\special{pa 5000 1600}%
\special{pa 4800 1600}%
\special{pa 4800 1400}%
\special{pa 5000 1400}%
\special{fp}%
%
\special{pn 0}%
\special{sh 0.150}%
\special{pa 5000 1400}%
\special{pa 5200 1400}%
\special{pa 5200 1600}%
\special{pa 5000 1600}%
\special{pa 5000 1400}%
\special{ip}%
\special{pn 8}%
\special{pa 5000 1400}%
\special{pa 5200 1400}%
\special{pa 5200 1600}%
\special{pa 5000 1600}%
\special{pa 5000 1400}%
\special{pa 5200 1400}%
\special{fp}%
%
\special{pn 0}%
\special{sh 0.150}%
\special{pa 5200 1400}%
\special{pa 5400 1400}%
\special{pa 5400 1600}%
\special{pa 5200 1600}%
\special{pa 5200 1400}%
\special{ip}%
\special{pn 8}%
\special{pa 5200 1400}%
\special{pa 5400 1400}%
\special{pa 5400 1600}%
\special{pa 5200 1600}%
\special{pa 5200 1400}%
\special{pa 5400 1400}%
\special{fp}%
%
\special{pn 0}%
\special{sh 0.150}%
\special{pa 5400 1400}%
\special{pa 5600 1400}%
\special{pa 5600 1600}%
\special{pa 5400 1600}%
\special{pa 5400 1400}%
\special{ip}%
\special{pn 8}%
\special{pa 5400 1400}%
\special{pa 5600 1400}%
\special{pa 5600 1600}%
\special{pa 5400 1600}%
\special{pa 5400 1400}%
\special{pa 5600 1400}%
\special{fp}%
%
\special{pn 8}%
\special{pa 3600 1600}%
\special{pa 3800 1600}%
\special{pa 3800 1800}%
\special{pa 3600 1800}%
\special{pa 3600 1600}%
\special{pa 3800 1600}%
\special{fp}%
%
\special{pn 8}%
\special{pa 3800 1600}%
\special{pa 4000 1600}%
\special{pa 4000 1800}%
\special{pa 3800 1800}%
\special{pa 3800 1600}%
\special{pa 4000 1600}%
\special{fp}%
%
\special{pn 8}%
\special{pa 4000 1600}%
\special{pa 4200 1600}%
\special{pa 4200 1800}%
\special{pa 4000 1800}%
\special{pa 4000 1600}%
\special{pa 4200 1600}%
\special{fp}%
%
\special{pn 0}%
\special{sh 0.150}%
\special{pa 4200 1600}%
\special{pa 4400 1600}%
\special{pa 4400 1800}%
\special{pa 4200 1800}%
\special{pa 4200 1600}%
\special{ip}%
\special{pn 8}%
\special{pa 4200 1600}%
\special{pa 4400 1600}%
\special{pa 4400 1800}%
\special{pa 4200 1800}%
\special{pa 4200 1600}%
\special{pa 4400 1600}%
\special{fp}%
%
\special{pn 0}%
\special{sh 0.150}%
\special{pa 4400 1600}%
\special{pa 4600 1600}%
\special{pa 4600 1800}%
\special{pa 4400 1800}%
\special{pa 4400 1600}%
\special{ip}%
\special{pn 8}%
\special{pa 4400 1600}%
\special{pa 4600 1600}%
\special{pa 4600 1800}%
\special{pa 4400 1800}%
\special{pa 4400 1600}%
\special{pa 4600 1600}%
\special{fp}%
%
\special{pn 0}%
\special{sh 0.150}%
\special{pa 4600 1600}%
\special{pa 4800 1600}%
\special{pa 4800 1800}%
\special{pa 4600 1800}%
\special{pa 4600 1600}%
\special{ip}%
\special{pn 8}%
\special{pa 4600 1600}%
\special{pa 4800 1600}%
\special{pa 4800 1800}%
\special{pa 4600 1800}%
\special{pa 4600 1600}%
\special{pa 4800 1600}%
\special{fp}%
%
\special{pn 0}%
\special{sh 0.150}%
\special{pa 4800 1600}%
\special{pa 5000 1600}%
\special{pa 5000 1800}%
\special{pa 4800 1800}%
\special{pa 4800 1600}%
\special{ip}%
\special{pn 8}%
\special{pa 4800 1600}%
\special{pa 5000 1600}%
\special{pa 5000 1800}%
\special{pa 4800 1800}%
\special{pa 4800 1600}%
\special{pa 5000 1600}%
\special{fp}%
%
\special{pn 0}%
\special{sh 0.150}%
\special{pa 5000 1600}%
\special{pa 5200 1600}%
\special{pa 5200 1800}%
\special{pa 5000 1800}%
\special{pa 5000 1600}%
\special{ip}%
\special{pn 8}%
\special{pa 5000 1600}%
\special{pa 5200 1600}%
\special{pa 5200 1800}%
\special{pa 5000 1800}%
\special{pa 5000 1600}%
\special{pa 5200 1600}%
\special{fp}%
%
\special{pn 0}%
\special{sh 0.150}%
\special{pa 5200 1600}%
\special{pa 5400 1600}%
\special{pa 5400 1800}%
\special{pa 5200 1800}%
\special{pa 5200 1600}%
\special{ip}%
\special{pn 8}%
\special{pa 5200 1600}%
\special{pa 5400 1600}%
\special{pa 5400 1800}%
\special{pa 5200 1800}%
\special{pa 5200 1600}%
\special{pa 5400 1600}%
\special{fp}%
%
\special{pn 0}%
\special{sh 0.150}%
\special{pa 5400 1600}%
\special{pa 5600 1600}%
\special{pa 5600 1800}%
\special{pa 5400 1800}%
\special{pa 5400 1600}%
\special{ip}%
\special{pn 8}%
\special{pa 5400 1600}%
\special{pa 5600 1600}%
\special{pa 5600 1800}%
\special{pa 5400 1800}%
\special{pa 5400 1600}%
\special{pa 5600 1600}%
\special{fp}%
%
\special{pn 8}%
\special{pa 3600 1800}%
\special{pa 3800 1800}%
\special{pa 3800 2000}%
\special{pa 3600 2000}%
\special{pa 3600 1800}%
\special{pa 3800 1800}%
\special{fp}%
%
\special{pn 8}%
\special{pa 3800 1800}%
\special{pa 4000 1800}%
\special{pa 4000 2000}%
\special{pa 3800 2000}%
\special{pa 3800 1800}%
\special{pa 4000 1800}%
\special{fp}%
%
\special{pn 8}%
\special{pa 4000 1800}%
\special{pa 4200 1800}%
\special{pa 4200 2000}%
\special{pa 4000 2000}%
\special{pa 4000 1800}%
\special{pa 4200 1800}%
\special{fp}%
%
\special{pn 8}%
\special{pa 4200 1800}%
\special{pa 4400 1800}%
\special{pa 4400 2000}%
\special{pa 4200 2000}%
\special{pa 4200 1800}%
\special{pa 4400 1800}%
\special{fp}%
%
\special{pn 0}%
\special{sh 0.150}%
\special{pa 4400 1800}%
\special{pa 4600 1800}%
\special{pa 4600 2000}%
\special{pa 4400 2000}%
\special{pa 4400 1800}%
\special{ip}%
\special{pn 8}%
\special{pa 4400 1800}%
\special{pa 4600 1800}%
\special{pa 4600 2000}%
\special{pa 4400 2000}%
\special{pa 4400 1800}%
\special{pa 4600 1800}%
\special{fp}%
%
\special{pn 0}%
\special{sh 0.150}%
\special{pa 4600 1800}%
\special{pa 4800 1800}%
\special{pa 4800 2000}%
\special{pa 4600 2000}%
\special{pa 4600 1800}%
\special{ip}%
\special{pn 8}%
\special{pa 4600 1800}%
\special{pa 4800 1800}%
\special{pa 4800 2000}%
\special{pa 4600 2000}%
\special{pa 4600 1800}%
\special{pa 4800 1800}%
\special{fp}%
%
\special{pn 0}%
\special{sh 0.150}%
\special{pa 4800 1800}%
\special{pa 5000 1800}%
\special{pa 5000 2000}%
\special{pa 4800 2000}%
\special{pa 4800 1800}%
\special{ip}%
\special{pn 8}%
\special{pa 4800 1800}%
\special{pa 5000 1800}%
\special{pa 5000 2000}%
\special{pa 4800 2000}%
\special{pa 4800 1800}%
\special{pa 5000 1800}%
\special{fp}%
%
\special{pn 0}%
\special{sh 0.150}%
\special{pa 5000 1800}%
\special{pa 5200 1800}%
\special{pa 5200 2000}%
\special{pa 5000 2000}%
\special{pa 5000 1800}%
\special{ip}%
\special{pn 8}%
\special{pa 5000 1800}%
\special{pa 5200 1800}%
\special{pa 5200 2000}%
\special{pa 5000 2000}%
\special{pa 5000 1800}%
\special{pa 5200 1800}%
\special{fp}%
%
\special{pn 0}%
\special{sh 0.150}%
\special{pa 5200 1800}%
\special{pa 5400 1800}%
\special{pa 5400 2000}%
\special{pa 5200 2000}%
\special{pa 5200 1800}%
\special{ip}%
\special{pn 8}%
\special{pa 5200 1800}%
\special{pa 5400 1800}%
\special{pa 5400 2000}%
\special{pa 5200 2000}%
\special{pa 5200 1800}%
\special{pa 5400 1800}%
\special{fp}%
%
\special{pn 0}%
\special{sh 0.150}%
\special{pa 5400 1800}%
\special{pa 5600 1800}%
\special{pa 5600 2000}%
\special{pa 5400 2000}%
\special{pa 5400 1800}%
\special{ip}%
\special{pn 8}%
\special{pa 5400 1800}%
\special{pa 5600 1800}%
\special{pa 5600 2000}%
\special{pa 5400 2000}%
\special{pa 5400 1800}%
\special{pa 5600 1800}%
\special{fp}%
%
\special{pn 8}%
\special{pa 3600 2000}%
\special{pa 3800 2000}%
\special{pa 3800 2200}%
\special{pa 3600 2200}%
\special{pa 3600 2000}%
\special{pa 3800 2000}%
\special{fp}%
%
\special{pn 8}%
\special{pa 3800 2000}%
\special{pa 4000 2000}%
\special{pa 4000 2200}%
\special{pa 3800 2200}%
\special{pa 3800 2000}%
\special{pa 4000 2000}%
\special{fp}%
%
\special{pn 8}%
\special{pa 4000 2000}%
\special{pa 4200 2000}%
\special{pa 4200 2200}%
\special{pa 4000 2200}%
\special{pa 4000 2000}%
\special{pa 4200 2000}%
\special{fp}%
%
\special{pn 8}%
\special{pa 4200 2000}%
\special{pa 4400 2000}%
\special{pa 4400 2200}%
\special{pa 4200 2200}%
\special{pa 4200 2000}%
\special{pa 4400 2000}%
\special{fp}%
%
\special{pn 8}%
\special{pa 4400 2000}%
\special{pa 4600 2000}%
\special{pa 4600 2200}%
\special{pa 4400 2200}%
\special{pa 4400 2000}%
\special{pa 4600 2000}%
\special{fp}%
%
\special{pn 0}%
\special{sh 0.150}%
\special{pa 4600 2000}%
\special{pa 4800 2000}%
\special{pa 4800 2200}%
\special{pa 4600 2200}%
\special{pa 4600 2000}%
\special{ip}%
\special{pn 8}%
\special{pa 4600 2000}%
\special{pa 4800 2000}%
\special{pa 4800 2200}%
\special{pa 4600 2200}%
\special{pa 4600 2000}%
\special{pa 4800 2000}%
\special{fp}%
%
\special{pn 0}%
\special{sh 0.150}%
\special{pa 4800 2000}%
\special{pa 5000 2000}%
\special{pa 5000 2200}%
\special{pa 4800 2200}%
\special{pa 4800 2000}%
\special{ip}%
\special{pn 8}%
\special{pa 4800 2000}%
\special{pa 5000 2000}%
\special{pa 5000 2200}%
\special{pa 4800 2200}%
\special{pa 4800 2000}%
\special{pa 5000 2000}%
\special{fp}%
%
\special{pn 0}%
\special{sh 0.150}%
\special{pa 5000 2000}%
\special{pa 5200 2000}%
\special{pa 5200 2200}%
\special{pa 5000 2200}%
\special{pa 5000 2000}%
\special{ip}%
\special{pn 8}%
\special{pa 5000 2000}%
\special{pa 5200 2000}%
\special{pa 5200 2200}%
\special{pa 5000 2200}%
\special{pa 5000 2000}%
\special{pa 5200 2000}%
\special{fp}%
%
\special{pn 0}%
\special{sh 0.150}%
\special{pa 5200 2000}%
\special{pa 5400 2000}%
\special{pa 5400 2200}%
\special{pa 5200 2200}%
\special{pa 5200 2000}%
\special{ip}%
\special{pn 8}%
\special{pa 5200 2000}%
\special{pa 5400 2000}%
\special{pa 5400 2200}%
\special{pa 5200 2200}%
\special{pa 5200 2000}%
\special{pa 5400 2000}%
\special{fp}%
%
\special{pn 0}%
\special{sh 0.150}%
\special{pa 5400 2000}%
\special{pa 5600 2000}%
\special{pa 5600 2200}%
\special{pa 5400 2200}%
\special{pa 5400 2000}%
\special{ip}%
\special{pn 8}%
\special{pa 5400 2000}%
\special{pa 5600 2000}%
\special{pa 5600 2200}%
\special{pa 5400 2200}%
\special{pa 5400 2000}%
\special{pa 5600 2000}%
\special{fp}%
%
\special{pn 8}%
\special{pa 3600 2200}%
\special{pa 3800 2200}%
\special{pa 3800 2400}%
\special{pa 3600 2400}%
\special{pa 3600 2200}%
\special{pa 3800 2200}%
\special{fp}%
%
\special{pn 8}%
\special{pa 3800 2200}%
\special{pa 4000 2200}%
\special{pa 4000 2400}%
\special{pa 3800 2400}%
\special{pa 3800 2200}%
\special{pa 4000 2200}%
\special{fp}%
%
\special{pn 8}%
\special{pa 4000 2200}%
\special{pa 4200 2200}%
\special{pa 4200 2400}%
\special{pa 4000 2400}%
\special{pa 4000 2200}%
\special{pa 4200 2200}%
\special{fp}%
%
\special{pn 8}%
\special{pa 4200 2200}%
\special{pa 4400 2200}%
\special{pa 4400 2400}%
\special{pa 4200 2400}%
\special{pa 4200 2200}%
\special{pa 4400 2200}%
\special{fp}%
%
\special{pn 8}%
\special{pa 4400 2200}%
\special{pa 4600 2200}%
\special{pa 4600 2400}%
\special{pa 4400 2400}%
\special{pa 4400 2200}%
\special{pa 4600 2200}%
\special{fp}%
%
\special{pn 8}%
\special{pa 4600 2200}%
\special{pa 4800 2200}%
\special{pa 4800 2400}%
\special{pa 4600 2400}%
\special{pa 4600 2200}%
\special{pa 4800 2200}%
\special{fp}%
%
\special{pn 8}%
\special{pa 4800 2200}%
\special{pa 5000 2200}%
\special{pa 5000 2400}%
\special{pa 4800 2400}%
\special{pa 4800 2200}%
\special{pa 5000 2200}%
\special{fp}%
%
\special{pn 0}%
\special{sh 0.150}%
\special{pa 5000 2200}%
\special{pa 5200 2200}%
\special{pa 5200 2400}%
\special{pa 5000 2400}%
\special{pa 5000 2200}%
\special{ip}%
\special{pn 8}%
\special{pa 5000 2200}%
\special{pa 5200 2200}%
\special{pa 5200 2400}%
\special{pa 5000 2400}%
\special{pa 5000 2200}%
\special{pa 5200 2200}%
\special{fp}%
%
\special{pn 0}%
\special{sh 0.150}%
\special{pa 5200 2200}%
\special{pa 5400 2200}%
\special{pa 5400 2400}%
\special{pa 5200 2400}%
\special{pa 5200 2200}%
\special{ip}%
\special{pn 8}%
\special{pa 5200 2200}%
\special{pa 5400 2200}%
\special{pa 5400 2400}%
\special{pa 5200 2400}%
\special{pa 5200 2200}%
\special{pa 5400 2200}%
\special{fp}%
%
\special{pn 0}%
\special{sh 0.150}%
\special{pa 5400 2200}%
\special{pa 5600 2200}%
\special{pa 5600 2400}%
\special{pa 5400 2400}%
\special{pa 5400 2200}%
\special{ip}%
\special{pn 8}%
\special{pa 5400 2200}%
\special{pa 5600 2200}%
\special{pa 5600 2400}%
\special{pa 5400 2400}%
\special{pa 5400 2200}%
\special{pa 5600 2200}%
\special{fp}%
%
\special{pn 8}%
\special{pa 3600 2400}%
\special{pa 3800 2400}%
\special{pa 3800 2600}%
\special{pa 3600 2600}%
\special{pa 3600 2400}%
\special{pa 3800 2400}%
\special{fp}%
%
\special{pn 8}%
\special{pa 3800 2400}%
\special{pa 4000 2400}%
\special{pa 4000 2600}%
\special{pa 3800 2600}%
\special{pa 3800 2400}%
\special{pa 4000 2400}%
\special{fp}%
%
\special{pn 8}%
\special{pa 4000 2400}%
\special{pa 4200 2400}%
\special{pa 4200 2600}%
\special{pa 4000 2600}%
\special{pa 4000 2400}%
\special{pa 4200 2400}%
\special{fp}%
%
\special{pn 8}%
\special{pa 4200 2400}%
\special{pa 4400 2400}%
\special{pa 4400 2600}%
\special{pa 4200 2600}%
\special{pa 4200 2400}%
\special{pa 4400 2400}%
\special{fp}%
%
\special{pn 8}%
\special{pa 4400 2400}%
\special{pa 4600 2400}%
\special{pa 4600 2600}%
\special{pa 4400 2600}%
\special{pa 4400 2400}%
\special{pa 4600 2400}%
\special{fp}%
%
\special{pn 8}%
\special{pa 4600 2400}%
\special{pa 4800 2400}%
\special{pa 4800 2600}%
\special{pa 4600 2600}%
\special{pa 4600 2400}%
\special{pa 4800 2400}%
\special{fp}%
%
\special{pn 8}%
\special{pa 4800 2400}%
\special{pa 5000 2400}%
\special{pa 5000 2600}%
\special{pa 4800 2600}%
\special{pa 4800 2400}%
\special{pa 5000 2400}%
\special{fp}%
%
\special{pn 8}%
\special{pa 5000 2400}%
\special{pa 5200 2400}%
\special{pa 5200 2600}%
\special{pa 5000 2600}%
\special{pa 5000 2400}%
\special{pa 5200 2400}%
\special{fp}%
%
\special{pn 0}%
\special{sh 0.150}%
\special{pa 5200 2400}%
\special{pa 5400 2400}%
\special{pa 5400 2600}%
\special{pa 5200 2600}%
\special{pa 5200 2400}%
\special{ip}%
\special{pn 8}%
\special{pa 5200 2400}%
\special{pa 5400 2400}%
\special{pa 5400 2600}%
\special{pa 5200 2600}%
\special{pa 5200 2400}%
\special{pa 5400 2400}%
\special{fp}%
%
\special{pn 0}%
\special{sh 0.150}%
\special{pa 5400 2400}%
\special{pa 5600 2400}%
\special{pa 5600 2600}%
\special{pa 5400 2600}%
\special{pa 5400 2400}%
\special{ip}%
\special{pn 8}%
\special{pa 5400 2400}%
\special{pa 5600 2400}%
\special{pa 5600 2600}%
\special{pa 5400 2600}%
\special{pa 5400 2400}%
\special{pa 5600 2400}%
\special{fp}%
%
\special{pn 8}%
\special{pa 3600 2600}%
\special{pa 3800 2600}%
\special{pa 3800 2800}%
\special{pa 3600 2800}%
\special{pa 3600 2600}%
\special{pa 3800 2600}%
\special{fp}%
%
\special{pn 8}%
\special{pa 3800 2600}%
\special{pa 4000 2600}%
\special{pa 4000 2800}%
\special{pa 3800 2800}%
\special{pa 3800 2600}%
\special{pa 4000 2600}%
\special{fp}%
%
\special{pn 8}%
\special{pa 4000 2600}%
\special{pa 4200 2600}%
\special{pa 4200 2800}%
\special{pa 4000 2800}%
\special{pa 4000 2600}%
\special{pa 4200 2600}%
\special{fp}%
%
\special{pn 8}%
\special{pa 4200 2600}%
\special{pa 4400 2600}%
\special{pa 4400 2800}%
\special{pa 4200 2800}%
\special{pa 4200 2600}%
\special{pa 4400 2600}%
\special{fp}%
%
\special{pn 8}%
\special{pa 4400 2600}%
\special{pa 4600 2600}%
\special{pa 4600 2800}%
\special{pa 4400 2800}%
\special{pa 4400 2600}%
\special{pa 4600 2600}%
\special{fp}%
%
\special{pn 8}%
\special{pa 4600 2600}%
\special{pa 4800 2600}%
\special{pa 4800 2800}%
\special{pa 4600 2800}%
\special{pa 4600 2600}%
\special{pa 4800 2600}%
\special{fp}%
%
\special{pn 8}%
\special{pa 4800 2600}%
\special{pa 5000 2600}%
\special{pa 5000 2800}%
\special{pa 4800 2800}%
\special{pa 4800 2600}%
\special{pa 5000 2600}%
\special{fp}%
%
\special{pn 8}%
\special{pa 5000 2600}%
\special{pa 5200 2600}%
\special{pa 5200 2800}%
\special{pa 5000 2800}%
\special{pa 5000 2600}%
\special{pa 5200 2600}%
\special{fp}%
%
\special{pn 0}%
\special{sh 0.150}%
\special{pa 5200 2600}%
\special{pa 5400 2600}%
\special{pa 5400 2800}%
\special{pa 5200 2800}%
\special{pa 5200 2600}%
\special{ip}%
\special{pn 8}%
\special{pa 5200 2600}%
\special{pa 5400 2600}%
\special{pa 5400 2800}%
\special{pa 5200 2800}%
\special{pa 5200 2600}%
\special{pa 5400 2600}%
\special{fp}%
%
\special{pn 0}%
\special{sh 0.150}%
\special{pa 5400 2600}%
\special{pa 5600 2600}%
\special{pa 5600 2800}%
\special{pa 5400 2800}%
\special{pa 5400 2600}%
\special{ip}%
\special{pn 8}%
\special{pa 5400 2600}%
\special{pa 5600 2600}%
\special{pa 5600 2800}%
\special{pa 5400 2800}%
\special{pa 5400 2600}%
\special{pa 5600 2600}%
\special{fp}%
%
\special{pn 20}%
\special{pa 4200 800}%
\special{pa 4200 2800}%
\special{fp}%
%
\special{pn 20}%
\special{pa 3600 1400}%
\special{pa 5600 1400}%
\special{fp}%
%
\special{pn 20}%
\special{pa 5600 2200}%
\special{pa 3600 2200}%
\special{fp}%
%
\special{pn 20}%
\special{pa 3600 2400}%
\special{pa 5600 2400}%
\special{fp}%
%
\special{pn 20}%
\special{pa 5200 2800}%
\special{pa 5200 800}%
\special{fp}%
%
\special{pn 20}%
\special{pa 5000 800}%
\special{pa 5000 2800}%
\special{fp}%
\end{picture}}%
\caption{The Hessenberg function $h_J$}
\label{pic: example of hJ}
\end{figure}
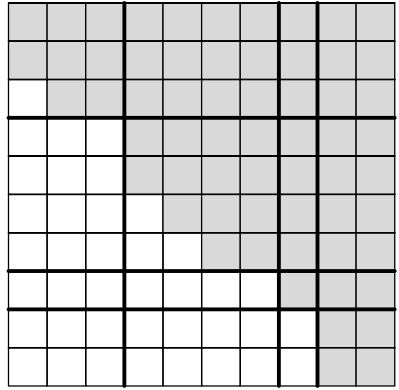
\end{example}


\section{The relation between $H^*(\Pet{n};\Z)$ and $H^*(\Perm{n};\Z)$}
The aim of this section is to prove that there is an isomorphism
\begin{align*}
 H^*(\Pet{n};\Z) \cong  H^*(\Perm{n};\Z)^{\Sn}
\end{align*} 
as graded rings.

\subsection{The Schubert varieties $X_{w_J}$ associated with $w_J$}

Let $J\subseteq[n-1]$. 
Recall that we have the decomposition $J=J_1\sqcup \cdots \sqcup J_m$ into the connected components.
For $1\le k\le m$, we define $\overline{J_k}\subseteq[n]$ by
\begin{align*}
 \overline{J_k} \coloneqq J_k\sqcup\{\max J_k +1\}.
\end{align*}
We also set
\begin{align*}
 n_k \coloneqq |\overline{J_k}|=|J_k|+1 \qquad \text{for $1\le k\le m$}.
\end{align*}
The permutation $w_J\in\Sn$ defined in \eqref{eq: def of wJ} determines the corresponding Schubert variety $X_{w_J}\subseteq Fl_n$. Since $w_J$ is a product of longest permutations of smaller ranks (see also Example~\ref{eg: wJ}), it follows that the associated Schubert variety $X_{w_J}$ is isomorphic to a product of flag varieties of smaller ranks: 
\begin{align}\label{eq: decomp of XwJ}
 X_{w_J} \cong \prod_{k=1}^{m} Fl_{n_k} .
\end{align}
Although this is well-known, let us construct an explicit isomorphism \eqref{eq: decomp of XwJ} to use it in the next subsection.
We begin with the map
\begin{align}\label{eq: emb onto image 2}
 \prod_{k=1}^{m} \text{GL}_{n_k}(\C) \rightarrow \text{GL}_n(\C)
 \quad ; \quad 
 (g_1,\cdots,g_m) \mapsto g_J ,
\end{align}
where $g_J$ is an $n\times n$ block-diagonal matrix defined as follows. 
For $1\le k\le m$, the 
$\overline{J_k}\times \overline{J_k}(\subseteq[n]\times [n])$ diagonal block of $g_J$ is $g_k$, and the remaining diagonal blocks are matrices of size 1 having $1$ as their entries.

\begin{example}\label{eg: matrix embedding 1}
{\rm
Let $n=10$ and $J=\{1,2,4,5,6,9\}=\{1,2\}\sqcup\{4,5,6\}\sqcup\{9\}$ as above.
Then we have
\begin{align*}
\text{$\overline{J_1}=\{1,2,3\}$, $\overline{J_2}=\{4,5,6,7\}$, and $\overline{J_3}=\{9,10\}$}
\end{align*}
so that $n_1=3$, $n_2=4$, and $n_3=2$.
The map \eqref{eq: emb onto image 2} sends an element $(g_1,g_2,g_3)\in \text{GL}_3(\C)\times \text{GL}_4(\C)\times \text{GL}_2(\C)$ to the block-diagonal matrix 
\begin{align*}
{\unitlength 0.1in%
\begin{picture}(28.0000,19.4000)(12.3000,-47.9000)%
%
\special{pn 0}%
\special{sh 0.150}%
\special{pa 1750 2860}%
\special{pa 2360 2860}%
\special{pa 2360 3420}%
\special{pa 1750 3420}%
\special{pa 1750 2860}%
\special{ip}%
\special{pn 8}%
\special{pa 1750 2860}%
\special{pa 2360 2860}%
\special{pa 2360 3420}%
\special{pa 1750 3420}%
\special{pa 1750 2860}%
\special{ip}%
%
\special{pn 0}%
\special{sh 0.150}%
\special{pa 2360 3420}%
\special{pa 3230 3420}%
\special{pa 3230 4210}%
\special{pa 2360 4210}%
\special{pa 2360 3420}%
\special{ip}%
\special{pn 8}%
\special{pa 2360 3420}%
\special{pa 3230 3420}%
\special{pa 3230 4210}%
\special{pa 2360 4210}%
\special{pa 2360 3420}%
\special{ip}%
%
\special{pn 0}%
\special{sh 0.150}%
\special{pa 3230 4200}%
\special{pa 3450 4200}%
\special{pa 3450 4400}%
\special{pa 3230 4400}%
\special{pa 3230 4200}%
\special{ip}%
\special{pn 8}%
\special{pa 3230 4200}%
\special{pa 3450 4200}%
\special{pa 3450 4400}%
\special{pa 3230 4400}%
\special{pa 3230 4200}%
\special{ip}%
%
\special{pn 0}%
\special{sh 0.150}%
\special{pa 3450 4400}%
\special{pa 3850 4400}%
\special{pa 3850 4780}%
\special{pa 3450 4780}%
\special{pa 3450 4400}%
\special{ip}%
\special{pn 8}%
\special{pa 3450 4400}%
\special{pa 3850 4400}%
\special{pa 3850 4780}%
\special{pa 3450 4780}%
\special{pa 3450 4400}%
\special{ip}%
\put(19.8000,-31.9000){\makebox(0,0)[lb]{$g_1$}}%
\put(27.4000,-38.7000){\makebox(0,0)[lb]{$g_2$}}%
\put(36.0000,-46.4000){\makebox(0,0)[lb]{$g_3$}}%
%
\special{pn 4}%
\special{pa 2360 2850}%
\special{pa 2360 4790}%
\special{fp}%
%
\special{pn 4}%
\special{pa 1750 3420}%
\special{pa 3850 3420}%
\special{fp}%
%
\special{pn 4}%
\special{pa 3230 2850}%
\special{pa 3230 4790}%
\special{fp}%
\put(33.0000,-43.5000){\makebox(0,0)[lb]{$1$}}%
%
\special{pn 4}%
\special{pa 1750 4200}%
\special{pa 3850 4200}%
\special{fp}%
%
\special{pn 4}%
\special{pa 1750 4400}%
\special{pa 3850 4400}%
\special{fp}%
%
\special{pn 4}%
\special{pa 3450 2850}%
\special{pa 3450 4790}%
\special{fp}%
\put(12.3000,-48.0000){\makebox(0,0)[lb]{$g_J=\left( \begin{array}{@{\,}cccccccccc@{\,}} & & \hspace{50pt} & & & & & & & \\ & \ & & & & & & & & \\ \ & & & & & & & & & \\ & & & & & & & & & \\ & & & & & & & & & \\ & & & & & & & & & \\ & & & & & & & & & \\ & & & & & & & & & \\ & & & & & & & & & \\ & & & & & & & & & \end{array} \right)$}}%
\put(40.3000,-39.3000){\makebox(0,0)[lb]{$\in \text{GL}_{10}(\C)$}}%
\end{picture}}%
\end{align*}
(cf.\ Example~\ref{eg: wJ}).}
\end{example}

\vspace{10pt}

Let $B_n\subseteq \text{GL}_n(\C)$ be the Borel subgroup consisting of upper triangular matrices.
We then have the standard identification $Fl_n=\text{GL}_n(\C)/B_n$ as is well-known. For $g\in\text{GL}_n(\C)$, we write $[g]=gB_n\in \text{GL}_n(\C)/B_n$ for simplicity.
It is clear that the map \eqref{eq: emb onto image 2} induces an embedding
\begin{align}\label{eq: emb onto image 3}
\Emb{J} \colon  \prod_{k=1}^{m} Fl_{n_k} \rightarrow Fl_n
\quad ; \quad 
([g_1],\cdots,[g_m]) \mapsto [g_J].
\end{align}
We now show that the image of $\Emb{J}$ coincides with the Schubert variety $X_{w_J}$.
For simplicity, we identify the permutation $w_J$ and the element of $Fl_n$ represented by its permutation matrix (see Example~\ref{eg: wJ}).
Under this identification, it is straightforward to see that the embedding $\Emb{J}$ sends $(w_0^{(J_1)},w_0^{(J_2)},\ldots,w_0^{(J_m)} )\in \prod_{k=1}^{m} Fl_{n_k}$ to $w_J\in Fl_n$. This means that the image of $\Emb{J}$ contains $w_J$.
It also follows from the definition that the image of $\Emb{J}$ is stable under the action of $B_n(\subseteq \text{GL}_n(\C))$, where $B_n$ acts on $Fl_n=\text{GL}_n(\C)/B$ by restricting the left multiplication of $\text{GL}_n(\C)$ on $\text{GL}_n(\C)/B$. Therefore, the image of $\Emb{J}$ is a $B$-stable (Zariski-)closed subset of $Fl_n$ containing $w_J$. This means that $X_{w_J}\subseteq \image \Emb{J}$. Since the product $\prod_{k=1}^{m} Fl_{n_k}$ is irreducible, so is the image of $\Emb{J}$. 
We also know that the dimensions of $X_{w_J}$ and $\image \Emb{J}$ coincide since
\begin{align*}
\dim_{\C} X_{w_J} = \ell(w_J) = \sum_{k=1}^m \ell(w_0^{(J_k)}) 
= \dim_{\C} \left( \prod_{k=1}^{m} Fl_{n_k} \right)= \dim_{\C} \image \Emb{J}, 
\end{align*}
where $w_0^{(J_k)}$ is the permutation appeared in \eqref{eq: def of wJ}.
Hence, we conclude that 
\begin{align*}
 X_{w_J} = \image \Emb{J} .
\end{align*}
Therefore, we verified that the map \eqref{eq: emb onto image 3} is an embedding onto the Schubert variety $X_{w_J}$. 
This gives us the isomorphism in \eqref{eq: decomp of XwJ}.

\vspace{10pt}

\subsection{Varieties associated with $J$}
For each $J\subseteq[n-1]$, we introduce varieties $Fl_{J}$, $\Pet{J}$, $\Perm{J}$ associated with $J$ in what follows.
First, we set
\begin{align}\label{eq: definition of FlJ}
 Fl_{J} \coloneqq X_{w_J} \cong \prod_{k=1}^m Fl_{n_k},
\end{align}
where the last isomorphism is given by \eqref{eq: decomp of XwJ}.

\begin{example}\label{eg: decomp of FlJ}
{\rm
Let $n=10$ and $J=\{1,2,4,5,6,9\}=\{1,2\}\sqcup\{4,5,6\}\sqcup\{9\}$ as above. Then we have 
\begin{align*}
 &Fl_{J} \cong Fl_{3}\times Fl_{4}\times Fl_{2}
\end{align*}
(cf.\ Example~\ref{eg: matrix embedding 1}).}
\end{example}

Recall that $h_J\colon [n]\rightarrow[n]$ is the Hessenberg function defined in \eqref{eq: def of hJ}.
Associated with $h_J$, we consider two varieties $\Hess(N,h_J)$ and $\Hess(S,h_J)$, where we note that $\Hess(S,h_J)$ is not connected when $J\neq [n-1]$ (\cite[Corollary~9]{De Mari-Procesi-Shayman} or \cite[Lemma~3.12]{Teff11}).
It is clear that the identity flag 
\begin{align*}
 \langle e_1 \rangle \subset \langle e_1,e_2 \rangle \subset \cdots \subset \langle e_1,e_2,\ldots,e_n \rangle = \C^n
\end{align*}
belongs to $\Hess(S,h_J)$ by definition.
We denote by $\Hess^*(S,h_J)$ the connected component of $\Hess(S,h_J)$ containing the identity flag. 
We set
\begin{align*}
 &\Pet{J} \coloneqq 
 \Hess(N,h_J) \subseteq Fl_n, \\
 &\Perm{J} \coloneqq 
 \Hess^*(S,h_J) \subseteq Fl_n.
\end{align*}
Recalling that $\Pet{n}=\Hess(N,h_2)$ and $\Perm{n}=\Hess(S,h_2)$ from \eqref{eq: def of Pet} and \eqref{eq: def of Perm}, 
it follows that 
\begin{align*}
\Pet{J}\subseteq \Pet{n} \quad \text{and} \quad \Perm{J}\subseteq \Perm{n}
\end{align*}
by \eqref{eq: ineq for hJ}.

\begin{lemma}\label{lem: property of PetJ and PermJ}
For $J\subseteq[n-1]$, the following hold.
\begin{itemize}
\item[(i)] $\Pet{J}$ and $\Perm{J}$ are irreducible.
\item[(ii)] $\dim_{\C}\Pet{J}=\dim_{\C}\Perm{J}=|J|$.
\end{itemize}
\end{lemma}

\begin{proof}
The irreducibility of $\Pet{J}(=\Hess(N,h_J))$ follows from \cite[Sect.\ 7]{an-ty}.
For $\Perm{J}$ $(=\Hess^*(S,h_J))$, it is non-singular (see section~\ref{subsect: permutohedral}) and connected so that it is irreducible. This proves the claim (i).

For the claim (ii), we have
\begin{align*}
\dim_{\C}\Pet{J}=\sum_{j=1}^n(h_J(j)-j)=\dim_{\C}\Perm{J}
\end{align*}
by \cite[Theorem~10.2]{so-ty} and \cite[Theorem~8]{De Mari-Procesi-Shayman} (see also \cite[Sect.~7]{an-ty}). It is clear that this value is equal to $|J|$ by the definition of $h_J$.
\end{proof}

We now use the embedding $\Emb{J}\colon  \prod_{k=1}^{m} Fl_{n_k}\rightarrow Fl_n$ given in \eqref{eq: emb onto image 3} to study the structure of $\Pet{J}$ and $\Perm{J}$ for $J\subseteq[n-1]$.
We begin with considering the image of $\prod_{k=1}^{m} \Pet{n_k}$ under $\Emb{J}$.
It follows from the construction of $\Emb{J}$ that an arbitrary element $V_{\bullet}\in \Emb{J}\left( \prod_{k=1}^{m} \Pet{n_k} \right)$ satisfies
\begin{align*}
NV_{j} \subseteq V_{h_J(j)} \qquad (1\le j\le n)
\end{align*}
(see also Example~\ref{eg: matrix embedding 1}). Namely, we have
\begin{align*}
\Emb{J}\left( \prod_{k=1}^{m} \Pet{n_k} \right) \subseteq \Hess(N,h_J) = \Pet{J}.
\end{align*}
Here, we know that $\dim_{\C} \Emb{J}\left( \prod_{k=1}^{m} \Pet{n_k} \right)$ is equal to $\dim_{\C} \Pet{J}$ since
\begin{align*}
\dim_{\C} \Emb{J}\left( \prod_{k=1}^{m} \Pet{n_k} \right) 
= \sum_{k=1}^m (n_k -1)
= \sum_{i=1}^n (h_J(i)-i)
= \dim_{\C} \Pet{J}
\end{align*}
by \eqref{eq: dim of Pet} and \eqref{eq: def of hJ}.
Since $\Pet{J}$ is irreducible, we obtain
\begin{align}\label{eq: pet product image}
\Emb{J}\left( \prod_{k=1}^{m} \Pet{n_k} \right) = \Pet{J}.
\end{align}
To obtain a similar result for $\Perm{J}$, recall that $\Perm{J}=\Hess^*(S,h_J)$ is the connected component of $\Hess(S,h_J)$ containing the identity flag.
We also recall that $\Perm{J}$ is irreducible from Lemma~\ref{lem: property of PetJ and PermJ}.
Also, it is clear that the image $\Emb{J}\left( \prod_{k=1}^{m} \Perm{n_k} \right)$ contains the identity flag in $Fl_n$.
Thus, by an argument similar to that above, we obtain
\begin{align}\label{eq: perm product image}
\Emb{J}\left( \prod_{k=1}^{m} \Perm{n_k} \right) = \Perm{J}.
\end{align}
Since the image of $\Emb{J}$ is $Fl_J(=X_{w_J})$, the equalities \eqref{eq: pet product image} and \eqref{eq: perm product image} imply the following claim.

\begin{lemma}\label{lem: PetJ and PermJ in FlJ}
For $J\subseteq[n-1]$, both of $\Pet{J}$ and $\Perm{J}$ are contained in $Fl_J$.
\end{lemma}

\vspace{10pt}

Since $\Emb{J}$ is an embedding, the equalities \eqref{eq: pet product image} and \eqref{eq: perm product image} also imply the following decompositions into products (cf. \cite[Theorem 4.5]{Dre1} and \cite[Proposition 3.13]{Teff11}):
\begin{align}\label{eq: decomposition into products}
 \Pet{J} \cong \prod_{k=1}^{m} \Pet{n_k} \quad \text{and} \quad
 \Perm{J} \cong \prod_{k=1}^{m} \Perm{n_k}.
\end{align}
It is clear from the construction that these decompositions are compatible with the one in \eqref{eq: definition of FlJ}.

\begin{example}
{\rm
If $n=10$ and $J=\{1,2,4,5,6,9\}=\{1,2\}\sqcup\{4,5,6\}\sqcup\{9\}$ as above, then we have 
\begin{align*}
 &\Pet{J} \cong \Pet{3}\times \Pet{4}\times \Pet{2}, \\ 
 &\Perm{J} \cong \Perm{3}\times \Perm{4}\times \Perm{2}
\end{align*}
which are compatible with the decomposition of $Fl_J$ given in Example~\ref{eg: decomp of FlJ}.}
\end{example}

\vspace{20pt}
The following is a direct implication of Proposition~\ref{prop: known things 4}.

\begin{lemma}\label{lem: degeneration}
For $J\subseteq [n-1]$, we have
\begin{align*}
 [\Pet{J}]=[\Perm{J}] \quad \text{in $H_*(Fl_n;\Z)$},
\end{align*}
where $[\Pet{J}]$ and $[\Perm{J}]$ are the cycles representing the subvarieties $\Pet{J}$ and $\Perm{J}$ in $Fl_{n}$, respectively.
\end{lemma}

\begin{proof}
Since $\Pet{J}$ and $\Perm{J}$ are both subvariety of $Fl_J(\subseteq Fl_n)$ by Lemma~\ref{lem: PetJ and PermJ in FlJ}, it suffices to prove the equality in $H^*(Fl_J;\Z)$.
The decomposition $Fl_J \cong \prod_{k=1}^m Fl_{n_k}$ given in \eqref{eq: definition of FlJ} induces an isomorphism
\begin{align*}
 \xi\colon  H_*(\textstyle{\prod_{k=1}^{m}} Fl_{n_k};\Z) \cong H_*(Fl_{J};\Z).
\end{align*}
By \cite[Example 1.10.2]{Fulton}, we also have an isomorphism 
\begin{align*}
 \PM\colon  \bigotimes_{k=1}^{m} H_*(Fl_{n_k};\Z) \stackrel{\cong}{\rightarrow}H_*(\textstyle{\prod_{k=1}^{m}} Fl_{n_k};\Z) 
\end{align*}
such that $\PM(\otimes_{k=1}^{m} [V_{k}])=[\prod_{k=1}^{m} V_{k}]$ for irreducible subvarieties $V_k\subseteq Fl_{n_k}$. 
By composing these two isomorphisms, we have
\begin{equation}\label{eq: cycle J}
\begin{split}
 &\xi\circ \PM(\otimes_{k=1}^{m} [\Pet{n_k}])=\xi([\textstyle{\prod_{k=1}^{m} \Pet{n_k}}])=[\Pet{J}], \\
 &\xi\circ \PM(\otimes_{k=1}^{m} [\Perm{n_k}])=\xi([\textstyle{\prod_{k=1}^{m} \Perm{n_k}}])=[\Perm{J}] 
\end{split}
\end{equation}
since the isomorphisms in \eqref{eq: decomposition into products} are compatible with the isomorphism $Fl_J \cong \prod_{k=1}^m Fl_{n_k}$.
By Proposition~\ref{prop: known things 4}, we have the following equalities:
\begin{align*}
 [\Pet{n_k}]=[\Perm{n_k}] \quad \text{in $H_*(Fl_{n_k};\Z)$} \qquad (1\le k\le m).
\end{align*}
Therefore, \eqref{eq: cycle J} implies that 
\begin{align*}
 [\Pet{J}]=[\Perm{J}] \quad \text{in $H_*(Fl_{J};\Z)$}.
\end{align*}
\end{proof}

\vspace{10pt}

\subsection{A proof of Theorem \ref{thm: A}}\label{subsec: Proof of Main Thm}

Let 
\begin{align*}
 i \colon  \Pet{n}\hookrightarrow Fl_n, \quad \text{and} \quad
 j \colon  \Perm{n}\hookrightarrow Fl_n
\end{align*}
be the inclusion maps. 
We recall the following claim from \cite{Insko}.

\begin{proposition}\label{prop: known things 1}
$($\cite[Theorem 17]{Insko}$)$
The induced map $i_*\colon H_*(\Pet{n};\Z)\rightarrow H_*(Fl_n;\Z)$ is an injective map whose image is a direct summand of $H_*(Fl_n;\Z)$. 
\end{proposition}

Recall from Proposition~\ref{prop: known things 0 Pet} that $H_*(\Pet{n};\Z)$ and $H^*(Fl_n;\Z)$ are torsion free. Thus, the restriction map $i^* \colon  H^*(Fl_n;\Z) \rightarrow  H^*(\Pet{n};\Z)$ on the cohomology groups is the dual map of $i_*$ in Proposition~\ref{prop: known things 1}.

\begin{corollary}\label{corollary: known things 3}
$($\cite{Insko}$)$
The restriction map $i^* \colon  H^*(Fl_n;\Z) \rightarrow  H^*(\Pet{n};\Z)$ is surjective.
\end{corollary}

\vspace{5pt}

We now prove the following which gives us Theorem~\ref{thm: A} in Section~\ref{sec: intro}.

\begin{theorem}\label{thm: integral isomorphism}
There exists a unique isomorphism
\begin{align*}
\varphi\colon  H^*(\Pet{n};\Z) \rightarrow  H^*(\Perm{n};\Z)^{\Sn}
\end{align*} 
as graded rings such that the following diagram commutes.
\begin{center}
\begin{picture}(160,65)
   \put(50,50){$H^*(Fl_n;\Z)$}
   \put(40,32){\rotatebox[origin=c]{-135}{$\overrightarrow{\qquad\ }$}}
   \put(40,33){$\footnotesize{\text{$i^*$}}$}
   \put(120,33){$\footnotesize{\text{$j^*$}}$}
   \put(100,31){\rotatebox[origin=c]{-45}{$\overrightarrow{\qquad\ }$}}
   \put(0,10){$H^*(\Pet{n};\Z)$}
   \put(67,10){\rotatebox[origin=c]{0}{$\overrightarrow{\qquad\ }$}}
   \put(78,15){$\footnotesize{\text{$\cong$}}$}
   \put(78,5){$\footnotesize{\text{$\varphi$}}$}
   \put(100,10){$H^*(\Perm{n};\Z)^{\Sn}$}
\end{picture} 
\vspace{-5pt}
\end{center}
\end{theorem}

\begin{proof}
The uniqueness of $\varphi$ follows from the commutativity of the diagram and the surjectivity of $i^*$ (Corollary~\ref{corollary: known things 3}).
We construct such an isomorphism $\varphi$.

Let us begin with studying the induced maps on the \textit{homology groups}: 
\begin{align*}
&i_* \colon   H_*(\Pet{n};\Z) \rightarrow  H_*(Fl_n;\Z), \\ 
&j_* \colon   H_*(\Perm{n};\Z) \rightarrow  H_*(Fl_n;\Z).
\end{align*}
We first claim that
\begin{align}\label{eq: first claim}
\text{Im}\ i_* = \text{Im}\ j_* \quad \text{in $H_*(Fl_n;\Z)$}.
\end{align}
Let us prove this in the following.
For this purpose, we take a particular basis of $H_*(\Pet{n};\Z)$ as follows.
For each $J\subseteq[n-1]$, we have a cycle $[\Pet{J}]$ in $H_*(\Pet{n};\Z)$, and it is shown in \cite[Proposition~3.4 and Proposition~4.1]{AHKZ} that these cycles form a $\Z$-basis of $H_*(\Pet{n};\Z)$:
\begin{align*}
 H_*(\Pet{n};\Z) = \bigoplus_{J\subseteq[n-1]} \Z [\Pet{J}].
\end{align*}
By Lemma~\ref{lem: degeneration}, we have $i_*[\Pet{J}]=j_*[\Perm{J}]\in \text{Im}\ j_*$ for $J\subseteq[n-1]$.
This implies that 
\begin{align}\label{eq: i* sub j*}
 \text{Im}\ i_* \subseteq \text{Im}\ j_* . 
\end{align}
Let us prove that $\text{Im}\ i_* = \text{Im}\ j_*$. 
Since the map $i_*$ is injective by Proposition~\ref{prop: known things 1}, it follows from \eqref{eq: i* sub j*} and Proposition~\ref{prop: known things 0 Pet} that
\begin{align}\label{eq: first ineq}
 2^{n-1} = \rank (\text{Im}\ i_* )
 \le \rank (\text{Im} \  j_*).
\end{align}
As for $i_*$, the dual map of $j_*$ is precisely the restriction map
\begin{align*}
 j^* \colon  H^*(Fl_n;\Z)\rightarrow H^*(\Perm{n};\Z)
\end{align*}
on the cohomology groups. For this map, we know from Proposition~\ref{prop: known things 2} that 
\begin{align}\label{eq: second ineq}
\rank (\text{Im} \ j^*) \le \rank H^*(\Perm{n};\Z)^{\Sn}\le 2^{n-1}.
\end{align}
Since $j^*$ is the dual map of $j_*$, we have 
\begin{align*}
 \rank (\text{Im} \ j_*)=\dim_{\Q} (\text{Im} \  j_*^{\Q}) = \dim_{\Q} (\text{Im} \  j^*_{\Q})=\rank (\text{Im} \  j^*),
\end{align*}
where the maps $j_*^{\Q}$ and $j^*_{\Q}$ are the homomorphisms $H_*(\Perm{n};\Q)\rightarrow H_*(Fl_n;\Q)$ and $H^*(Fl_n;\Q)\rightarrow H^*(\Perm{n};\Q)$ induced by the inclusion map $j\colon \Perm{n}\rightarrow Fl_n$, respectively.
Thus, the inequalities in \eqref{eq: first ineq} and \eqref{eq: second ineq} must be equalities, and we obtain $\rank (\text{Im} \ i_*)=\rank (\text{Im} \ j_*)$. Hence, by \eqref{eq: i* sub j*} and Proposition~\ref{prop: known things 1}, it follows that $\text{Im}\ i_* = \text{Im}\ j_*$ 
as we claimed in \eqref{eq: first claim}.

Since $i_*\colon H_*(\Pet{n};\Z) \rightarrow H_*(Fl_n;\Z) $ is an isomorphism onto its image, 
\eqref{eq: first claim} means that there exists a surjective group homomorphism 
\begin{align*}
\psi\colon  H_*(\Perm{n};\Z) \rightarrow  H_*(\Pet{n};\Z) 
\end{align*}
which satisfies the following commutative diagram.
\begin{equation*}
\begin{split}
\begin{picture}(160,65)
   \put(50,50){$H_*(Fl_n;\Z)$}
   \put(40,32){\rotatebox[origin=c]{-135}{$\overleftarrow{\qquad\ }$}}
   \put(40,33){$\footnotesize{\text{$i_*$}}$}
   \put(120,33){$\footnotesize{\text{$j_*$}}$}
   \put(100,31){\rotatebox[origin=c]{-45}{$\overleftarrow{\qquad\ }$}}
   \put(0,10){$H_*(\Pet{n};\Z)$}
   \put(67,10){\rotatebox[origin=c]{0}{$\overleftarrow{\qquad\ }$}}
   \put(64,10){\rotatebox[origin=c]{0}{$\overleftarrow{\qquad\ }$}}
   \put(78,3){$\footnotesize{\text{$\psi$}}$}
   \put(100,10){$H_*(\Perm{n};\Z)$}
\end{picture} 
\end{split}
\end{equation*}
Now, we consider the following commutative diagram on the cohomology groups, where we denote by $\psi^*$ the dual map of $\psi$.
\begin{center}
\begin{picture}(160,65)
   \put(50,50){$H^*(Fl_n;\Z)$}
   \put(40,32){\rotatebox[origin=c]{-135}{$\overrightarrow{\qquad\ }$}}
   \put(40,33){$\footnotesize{\text{$i^*$}}$}
   \put(120,33){$\footnotesize{\text{$j^*$}}$}
   \put(100,31){\rotatebox[origin=c]{-45}{$\overrightarrow{\qquad\ }$}}
   \put(0,10){$H^*(\Pet{n};\Z)$}
   \put(67,10){\rotatebox[origin=c]{0}{$\overrightarrow{\qquad\ }$}}
   \put(78,3){$\footnotesize{\text{$\psi^*$}}$}
   \put(100,10){$H^*(\Perm{n};\Z)$}
\end{picture} 
\end{center}
Since $i^*$ is surjective by Corollary~\ref{corollary: known things 3}, we have $\text{Im} \ \psi^*=\text{Im} \ j^*$ by the commutativity of this diagram. 
Also, we know from Proposition~\ref{prop: known things 2}~(i) that $\text{Im} \ j^*\subseteq H^*(\Perm{n};\Z)^{\Sn}$. 
Thus, we obtain the following commutative diagram.
\begin{center}
\begin{picture}(160,65)
   \put(50,50){$H^*(Fl_n;\Z)$}
   \put(40,32){\rotatebox[origin=c]{-135}{$\overrightarrow{\qquad\ }$}}
   \put(40,33){$\footnotesize{\text{$i^*$}}$}
   \put(120,33){$\footnotesize{\text{$j^*$}}$}
   \put(100,31){\rotatebox[origin=c]{-45}{$\overrightarrow{\qquad\ }$}}
   \put(0,10){$H^*(\Pet{n};\Z)$}
   \put(67,10){\rotatebox[origin=c]{0}{$\overrightarrow{\qquad\ }$}}
   \put(78,3){$\footnotesize{\text{$\psi^*$}}$}
   \put(100,10){$H^*(\Perm{n};\Z)^{\Sn}$}
\end{picture} 
\end{center}
Since the map $i^*$ is a surjective ring homomorphism, it is the quotient map by an ideal of $H^*(Fl_n;\Z)$. Hence, it follows that the map $\psi^*$ must be a ring homomorphism by the commutativity of this diagram. Namely, $\psi^*$ is the (graded) ring homomorphism induced by $j^*$.
Since $\psi$ is surjective, the dual map $\psi^*$ is an injective map whose image is a direct summand of $H^*(\Perm{n};\Z)^{\Sn}$. 
Thus, the quotient $H^*(\Perm{n};\Z)^{\Sn}/\text{Im} \ \psi^*$ is a free $\Z$-module, and its rank is less than or equal to $0$ by Proposition~\ref{prop: known things 0 Pet}~(i) and Proposition~\ref{prop: known things 2}~(ii).
Therefore, it follows that $\text{Im} \ \psi^*=H^*(\Perm{n};\Z)^{\Sn}$ so that $\psi^*$ is surjective.
Letting $\varphi\coloneqq \psi^*$, we complete the proof.
\end{proof}

\vspace{10pt}

\begin{remark}\label{rem: on rank inequality}
\textnormal{
Since we proved that $\psi^*$ is an isomorphism, it follows that the equality 
\begin{align*}
\rank H_*(\Perm{n};\Z)^{\Sn} = 2^{n-1}
\end{align*}
holds for the inequality of Proposition~\ref{prop: known things 2}~(ii).
}
\end{remark}

\vspace{10pt}

\begin{remark}\label{rem: on Klyachko}
\textnormal{
The invariant subring $H^*(\Perm{n};\Q)^{\Sn}$ with $\Q$ coefficients was studied by Klyachko (\cite{Klyachko}), and he gave an explicit presentation of $H^*(\Perm{n};\Q)^{\Sn}$ (for arbitrary Lie types). See Nadeau-Tewari (\cite[Sect.~8]{Nadeau-Tewari}) for an exposition of Klyachko's results. One can verify that it coincides with the presentation of $H^*(\Pet{n};\Q)$ given by Fukukawa-Harada-Masuda (\cite{fu-ha-ma}). See \cite{AHMMS,ha-ho-ma} for a generalization to arbitrary Lie types.
}
\end{remark}

\bigskip

\section{An explicit presentation of the ring $H^*(\Pet{n};\Z)$}\label{sec: ring presentation}

The aim of this section is to give an explicit presentation of the ring $H^*(\Pet{n};\Z)$ in terms of ring generators and their relations.

\subsection{A ring presentation of $H^*(Fl_n;\Z)$}\label{subsec: coh of flag}
We review the following well-known presentation of the cohomology ring of the flag variety $Fl_n$. Our main reference is \cite[Sect.\ 10.2]{fult97}. For $1\le i\le n$, let $E_i$ be the tautological vector bundle over $Fl_n$ whose fiber over a point $V_{\bullet}\in Fl_n$ is $V_i$. As a convention, let $E_0$ be the sub-bundle of $E_1$ of rank $0$. Set 
\begin{align*}
 \tau_i \coloneqq c_1((E_i/E_{i-1})^*)
 \in H^2(Fl_n;\Z) 
 \qquad (1\le i\le n),
\end{align*}
where $c_1((E_i/E_{i-1})^*)$ is the first Chern class of the dual line bundle of the tautological line bundle\footnote{The line bundle $L_i$ appeared in Section \ref{sec: intro} is $E_i/E_{i-1}$ for $1\le i\le n$.} $E_i/E_{i-1}$.
By definition, we have short exact sequences
\begin{align*}
 0\rightarrow (E_i/E_{i-1})^* \rightarrow E_i^* \rightarrow E_{i-1}^*\rightarrow0
 \qquad (1\le i\le n).
\end{align*}
From these sequences, it follows that 
\begin{align*}
 c_k(E_n^*) = e_k(\tau_1,\tau_2,\ldots,\tau_n)
 \qquad (1\le k\le n),
\end{align*}
where $e_k(\tau_1,\tau_2,\ldots,\tau_n)$ is the $k$-th elementary symmetric polynomial in $\tau_1,\tau_2,\ldots,\tau_n$.
Since $E_n^*$ is a trivial bundle of rank $n$, this implies that we have 
\begin{align}\label{eq: ei=0 in flag}
 e_k(\tau_1,\tau_2,\ldots,\tau_n) = 0 \quad \text{in $H^*(Fl_n;\Z)$}
 \qquad (1\le k\le n).
\end{align}
Let $\Z[y_1,y_2,\ldots,y_n]$ be the polynomial ring over $\Z$ with indeterminates $y_1,y_2,\ldots,y_n$. 
The ring $H^*(Fl_n;\Z)$ is generated by $\tau_1,\tau_2,\ldots,\tau_n$, and hence we have a surjective ring homomorphism 
\begin{align*}
 \Z[y_1,y_2,\ldots,y_n] \rightarrow H^*(Fl_n;\Z)
\end{align*}
which sends $y_i$ to $\tau_i$ $(1\le i\le n)$.
By \eqref{eq: ei=0 in flag}, this induces a surjective ring homomorphism 
\begin{align}\label{eq: presentation for coho of flag}
 \Z[y_1,y_2,\ldots,y_n]/(e_1(y),e_2(y),\ldots,e_n(y)) \rightarrow H^*(Fl_n;\Z),
\end{align}
where $(e_1(y),e_2(y),\ldots,e_n(y))$ is the ideal of $\Z[y_1,y_2,\ldots,y_n]$ generated by $e_k(y)=e_k(y_1,y_2,\ldots,y_n)$ for $1\le k\le n$.
It is well-known that \eqref{eq: presentation for coho of flag} is an isomorphism.

\subsection{A module basis of $H^*(\Pet{n};\Z)$}\label{subsec: basis of Pet}
We set
\begin{align}\label{eq: def of x i}
 x_i \coloneqq c_1((E_i/E_{i-1})^*|_{\Pet{n}}) \in H^2(\Pet{n};\Z) 
 \qquad (1\le i\le n).
\end{align}
Namely, $x_i$ is the image of $\tau_i$ under the restriction map $i^*\colon H^*(Fl_n;\Z)\rightarrow H^*(\Pet{n};\Z)$ for $1\le i\le n$.
We also set
\begin{align}\label{eq: def of varpi i}
 \varpi_i \coloneqq x_1+x_2+\cdots+x_i \in H^2(\Pet{n};\Z) 
 \qquad (1\le i\le n-1).
\end{align}
For $J\subseteq[n-1]$, we have the decomposition $J=J_1\sqcup\cdots\sqcup J_m$ into the connected components (see section~\ref{subsec: comb on J}), and we set 
\begin{align}\label{eq: def of varpiJ}
 \varpi_J \coloneqq \frac{1}{m_J} \prod_{i\in J} \varpi_i,
\end{align}
where $m_J$ is the positive integer defined by $m_J \coloneqq |J_1|! |J_2|! \cdots |J_m|!$.
This $\varpi_J$ is defined in $H^{2|J|}(\Pet{n};\Q)$, but we have the following theorem.

\begin{theorem}\label{thm: main thm of AHKZ}
$($\cite[Theorem~4.13]{AHKZ}$)$
For each $J\subseteq[n-1]$, the cohomology class $\varpi_J$ is an element of the integral cohomology group $H^{2|J|}(\Pet{n};\Z)$, and the set 
\begin{align*}
\{\varpi_J \in H^{2|J|}(\Pet{n};\Z) \mid J\subseteq[n-1]\}
\end{align*}
is a $\Z$-basis of $H^*(\Pet{n};\Z)$.
\end{theorem}

\vspace{10pt}

In what follows, we give an integral expression for $\varpi_J$ in terms of $x_1,x_2,\ldots,x_n$ (see Corollary~\ref{cor: presentation 20} below).
For this purpose, we prepare two technical lemmas.

\begin{lemma}\label{lem: presentation 10}
For $1\le i\le n-1$, we have
\begin{align*}
 x_1^{d+1} + x_2^{d+1} + \cdots + x_i^{d+1}
 =  
 (x_1^d + x_2^d + \cdots + x_i^d )x_{i+1} 
 \qquad (d=1,2,\ldots) .
\end{align*} 
\end{lemma}

\begin{proof}
We first prove the claim for the case $d=1$ with $1\le i\le n-1$.
Recall from \cite[Lemma~4.7]{AHKZ} (cf.\ \cite{fu-ha-ma} and \cite{ha-ho-ma}) that
we have
\begin{align}\label{eq: alphapi=0}
 \alpha_j \varpi_j= 0 \qquad (1\le j\le n-1),
\end{align} 
where $\alpha_j\coloneqq x_j-x_{j+1}$.
Since we have $\varpi_j=\varpi_{j-1}+x_j$ by definition (with the convention $\varpi_0=0$), we have from \eqref{eq: alphapi=0} that
\begin{align*}
 (x_j - x_{j+1} ) (\varpi_{j-1} + x_j )= 0 \qquad (1\le j\le n-1).
\end{align*} 
From this, we obtain
\begin{equation}
\begin{split}\label{eq: elem 60}
 x_j^2 
 &= (\varpi_{j-1} + x_j )x_{j+1} - \varpi_{j-1}x_j \\
 &= \varpi_j x_{j+1} - \varpi_{j-1}x_j  \qquad \text{(by $\varpi_j=\varpi_{j-1}+x_j$ again)}
\end{split}
\end{equation}
for $1\le j\le n-1$.
Thus, we obtain
\begin{align*}
 x_1^2 + x_2^2 + \cdots + x_i^2 
 &= (\varpi_1 x_{2} - \varpi_{0}x_1) + (\varpi_2 x_3 - \varpi_1x_2) + \cdots + (\varpi_i x_{i+1} - \varpi_{i-1}x_i) \\
 &= - \varpi_{0}x_1 + \varpi_i x_{i+1} \\
 &= (x_1 + x_2 + \cdots + x_i) x_{i+1}
\end{align*} 
which gives the claim for $d=1$.

We assume that $d\ge2$ in what follows, and we prove the claim of this lemma by induction on $d$.
Assume by induction that 
\begin{align}\label{eq: ind hyp}
 x_1^{\ell+1} + x_2^{\ell+1} + \cdots + x_i^{\ell+1}
 =  
 (x_1^{\ell} + x_2^{\ell} + \cdots + x_i^{\ell} )x_{i+1} 
  \qquad (1\le \ell\le d-1),
\end{align} 
and we prove the claim for the case $\ell=d$.
To begin with, notice that
\begin{align*}
 x_j^{d+1} 
 = x_j^2\cdot x_j^{d-1}
 = \varpi_j x_j^{d-1} x_{j+1} - \varpi_{j-1} x_j^d
 \qquad (1\le j\le n-1)
\end{align*} 
by \eqref{eq: elem 60}.
Thus, we have
\begin{align}\label{eq: elem 70}
 x_1^{d+1} + x_2^{d+1} + \cdots + x_i^{d+1}
 = \sum_{j=1}^i \varpi_j x_j^{d-1} x_{j+1} - \sum_{j=1}^i \varpi_{j-1} x_j^d.
\end{align} 
The second summand in the right hand side can be written as
\begin{align*}
 -\sum_{j=1}^i \varpi_{j-1} x_j^d 
 = -\sum_{j=0}^{i-1} \varpi_{j} x_{j+1}^d 
 = -\sum_{j=1}^{i} \varpi_j x_{j+1}^d + \varpi_i x_{i+1}^d
\end{align*} 
since $\varpi_0=0$ by convention.
Applying this to \eqref{eq: elem 70}, we obtain
\begin{align}\label{eq: elem 170}
 x_1^{d+1} + x_2^{d+1} + \cdots + x_i^{d+1}
 = \sum_{j=1}^i \varpi_j ( x_j^{d-1}- x_{j+1}^{d-1}) x_{j+1} + \varpi_i x_{i+1}^d .
\end{align} 
In this equality, the first summand in the right hand side vanishes since it can be computed as
\begin{align*}
 \sum_{j=1}^i \varpi_j ( x_j^{d-1}- x_{j+1}^{d-1}) x_{j+1} 
 &= \sum_{j=1}^i \varpi_j (x_j - x_{j+1}) (x_j^{d-2}+x_j^{d-3}x_{j+1}+\cdots+x_{j+1}^{d-2}) x_{j+1} \\
 &=0 \qquad \text{(by \eqref{eq: alphapi=0})},
\end{align*}
where we take the convention $x_j^{d-2}+x_j^{d-3}x_{j+1}+\cdots+x_{j+1}^{d-2}=1$ when $d=2$.
Therefore, \eqref{eq: elem 170} and $\varpi_i=x_1+x_2+\cdots+x_i$ imply that
\begin{align*}
 x_1^{d+1} + x_2^{d+1} + \cdots + x_i^{d+1}
 = (x_1+x_2+\cdots+x_i) x_{i+1}^d .
\end{align*} 
Applying the inductive hypothesis \eqref{eq: ind hyp} to the right hand side repeatedly, we obtain
\begin{align*}
 x_1^{d+1} + x_2^{d+1} + \cdots + x_i^{d+1}
 &= (x_1 + x_2 + \cdots + x_i ) x_{i+1}^d \\
 &= (x_1^2 + x_2^2 + \cdots + x_i^2 )x_{i+1}^{d-1}\\
 &= \cdots \\
 &= (x_1^d + x_2^d + \cdots + x_i^d)x_{i+1}
\end{align*} 
which gives \eqref{eq: ind hyp} for the case $\ell=d$, as desired.
\end{proof}

\vspace{10pt}
For $1\le i\le n$ and a partition $\lambda=(\lambda_1,\lambda_2,\ldots)$ consisting of a weakly decreasing sequence of non-negative integers, let $m_{\lambda}(x_1,x_2\ldots,x_i)$ be the monomial symmetric polynomial in $x_1,x_2\ldots,x_i \in H^2(\Pet{n};\Z)$. That is, $m_{\lambda}(x_1,x_2\ldots,x_i)$ is the sum of all distinct monomials obtained from $x_1^{\lambda_1}x_2^{\lambda_2}\cdots x_i^{\lambda_i}$ by permuting the indices (e.g., \cite[Sect.~6.1]{fult97}).
For a positive integer $d$ and a non-negative integer $k$, 
we denote by $(d,1^k)$ the partition given by
\begin{align*}
 (d,1^k) = (d,\underbrace{1,1,\ldots,1}_{k},0,0,\cdots).
\end{align*} 
For $1\le k<i\le n$, we have the following identity:
\begin{equation}\label{eq: fund identity}
\begin{split}
 &(x_1^d + \cdots + x_i^d ) e_k(x_1,x_2,\ldots x_i)\\
 &\hspace{30pt}= m_{d+1,1^{k-1}}(x_1,x_2,\ldots,x_i) + m_{d,1^k}(x_1,x_2,\ldots,x_i) 
 \qquad (d=2,3,\ldots).
\end{split}
\end{equation}
One may obtain this identity by expanding the product in the left hand side and rearranging terms, without using any non-trivial algebraic relations for $x_1,x_2,\ldots,x_n$ in $H^*(\Pet{n};\Z)$. The following example illustrates the idea of the proof.

\begin{example}\label{ex: fond identity}
{\rm 
Let $d=5$, $i=4$, and $k=3$. Then we have
\begin{align*}
&(x_1^5 + x_2^5 + x_3^5 + x_4^5) e_3(x_1,x_2,x_3,x_4) \\
&\hspace{20pt}= 
(x_1^5 + x_2^5 + x_3^5 + x_4^5)\cdot (x_1x_2x_3 + x_1x_2x_4 + x_1x_3x_4 + x_2x_3x_4 ) \\
&\hspace{20pt}= 
m_{6,1,1}(x_1,x_2,x_3,x_4) + m_{5,1,1,1}(x_1,x_2,x_3,x_4),
\end{align*}
where we have
\begin{align*}
m_{6,1,1}(x_1,x_2,x_3,x_4)
&= x_1^6x_2x_3 + x_1^6x_2x_4 + x_1^6x_3x_4 
    + x_2^6x_1x_3 + x_2^6x_1x_4 + x_2^6x_3x_4\\
    &\hspace{30pt}
    + x_3^6x_1x_2 + x_3^6x_1x_4 + x_3^6x_2x_4
    + x_4^6x_1x_2 + x_4^6x_1x_3 + x_4^6x_2x_3
\end{align*}
and
\begin{align*}
m_{5,1,1,1}(x_1,x_2,x_3,x_4)
&= x_1^5x_2x_3x_4
    + x_2^5x_1x_3x_4 
    + x_3^5x_1x_2x_4
    + x_4^5x_1x_2x_3.
\end{align*}
}
\end{example}

\vspace{20pt}

We note that we have a slightly different identity when $d=1$. Namely, we have
\begin{equation}\label{eq: fund identity 2}
\begin{split}
 &(x_1 + \cdots + x_i ) e_k(x_1,x_2,\ldots x_i) \\
 &\hspace{30pt}= m_{2,1^{k-1}}(x_1,x_2,\ldots,x_i) + (k+1) m_{1^{k+1}}(x_1,x_2,\ldots,x_i) 
\end{split}
\end{equation} 
for $1\le k<i\le n$. Similarly to \eqref{eq: fund identity}, this identity can be obtained by expanding the product in the left hand side as illustrated in the following example.

\begin{example}\label{ex: fond identity 2}
{\rm 
Let $i=4$ and $k=2$. Then we have 
\begin{align*}
&(x_1 + x_2 + x_3 + x_4) e_2(x_1,x_2,x_3,x_4)\\
&\hspace{30pt}
= (x_1 + x_2 + x_3 + x_4) (x_1x_2 + x_1x_3 + x_1x_4 + x_2x_3 + x_2x_4 + x_3x_4)
\end{align*}
which is equal to the sum of 
\begin{align*}
m_{2,1}(x_1,x_2,x_3,x_4)
&= x_1^2x_2 + x_1^2x_3 + x_1^2x_4 
    + x_2^2x_1 + x_2^2x_3 + x_2^2x_4\\
    &\hspace{30pt}
    + x_3^2x_1 + x_3^2x_2 + x_3^2x_4
    + x_4^2x_1 + x_4^2x_2 + x_4^2x_3
\end{align*}
and
\begin{align*}
    &x_1(x_2x_3 + x_3x_4 + x_2x_4)
    + x_2(x_1x_3 + x_1x_4 + x_3x_4)\\
    &\hspace{30pt}+ x_3(x_1x_2 + x_1x_4 + x_2x_4)
    + x_4(x_1x_2 + x_1x_3 + x_2x_3)
    = 3 m_{1,1,1}(x_1,x_2,x_3,x_4).
\end{align*}
}
\end{example}

\vspace{20pt}

The next claim generalizes the previous lemma.
\begin{lemma}\label{lem: presentation 20}
For $0\le k< i\le n-1$, we have
\begin{align*}
 m_{d+1,1^k}(x_1,x_2,\ldots,x_i)
 = 
 \begin{cases} 
 m_{d,1^k}(x_1,x_2,\ldots,x_i) x_{i+1}  \quad &(d\ge2), \\\\
 (k+1) m_{1^{k+1}}(x_1,x_2,\ldots,x_i) x_{i+1} &(d=1).
 \end{cases}
\end{align*} 
\end{lemma}

\begin{proof}
When $k=0$, the claim is
\begin{align*}
 m_{d+1,1^0}(x_1,x_2,\ldots,x_i)
 = 
 \begin{cases} 
 m_{d,1^0}(x_1,x_2,\ldots,x_i) x_{i+1}  \quad &(d\ge2), \\\\
 m_{1^{1}}(x_1,x_2,\ldots,x_i) x_{i+1} &(d=1)
 \end{cases}
\end{align*} 
which is equivalent to 
\begin{align}\label{eq: elem 10}
x_1^{d+1} + x_2^{d+1} + \cdots + x_i^{d+1}
=  
(x_1^d + x_2^d + \cdots + x_i^d )x_{i+1} 
\quad  (d\ge1).
\end{align} 
Thus, the claim follows by the previous lemma when $k=0$.

We assume that $1\le k<i$ in what follows, and we prove the claim of this lemma by induction on $k$.
Assume by induction that
\begin{align}\label{eq: elem 40}
 m_{d+1,1^{k-1}}(x_1,x_2,\ldots,x_i) 
 = 
 \begin{cases} 
  m_{d,1^{k-1}}(x_1,x_2,\ldots,x_i) x_{i+1} \quad &(d\ge2), \\\\
  k m_{1^{k}}(x_1,x_2,\ldots,x_i) x_{i+1} &(d=1).
 \end{cases}
\end{align} 
Multiplying $e_k(x_1,x_2,\ldots,x_i)$ to the both sides of \eqref{eq: elem 10}, we have
\begin{equation}
\begin{split}\label{eq: from k=0}
  &(x_1^{d+1} + x_2^{d+1} + \cdots + x_i^{d+1})e_k(x_1,x_2,\ldots,x_i) \\
  &\hspace{30pt}= 
  (x_1^d + x_2^d + \cdots + x_i^d )e_k(x_1,x_2,\ldots,x_i)x_{i+1}.
\end{split}
\end{equation} 
By \eqref{eq: fund identity},
the left hand side of \eqref{eq: from k=0} is equal to
\begin{align}\label{eq: elem 30}
 m_{d+2,1^{k-1}}(x_1,x_2,\ldots,x_i) + m_{d+1,1^k}(x_1,x_2,\ldots,x_i)
\end{align}
since $d+1\ge2$.
By \eqref{eq: fund identity} and  \eqref{eq: fund identity 2}, the right hand side of \eqref{eq: from k=0} is equal to 
\begin{align}\label{eq: elem 20}
 \begin{cases} 
 m_{d+1,1^{k-1}}(x_1,x_2,\ldots,x_i) x_{i+1} + m_{d,1^k}(x_1,x_2,\ldots,x_i) x_{i+1} \quad &(d\ge2), \\\\
 m_{2,1^{k-1}}(x_1,x_2,\ldots,x_i) x_{i+1} + (k+1)m_{1^{k+1}}(x_1,x_2,\ldots,x_i) x_{i+1} &(d=1).
 \end{cases}
\end{align}
In both cases of $d\ge2$ and $d=1$, the first summands in \eqref{eq: elem 30} and \eqref{eq: elem 20} coincide because of the first case of the inductive hypothesis \eqref{eq: elem 40}.
Thus, the equality \eqref{eq: from k=0} implies that
\begin{align*}
 m_{d+1,1^k}(x_1,x_2,\ldots,x_i)
 = 
 \begin{cases} 
 m_{d,1^k}(x_1,x_2,\ldots,x_i) x_{i+1}  \quad &(d\ge2), \\\\
 (k+1) m_{1^{k+1}}(x_1,x_2,\ldots,x_i) x_{i+1} &(d=1),
 \end{cases}
\end{align*} 
as desired.
\end{proof}

\vspace{10pt}

For the next proposition, we recall that $\varpi_i=x_1+x_2+\cdots+x_i$ for $1\le i\le n-1$ from \eqref{eq: def of varpi i}.

\begin{proposition}\label{prop: presentation 10}
For $1\le a\le b\le n-1$, we have
\begin{align*}
\underbrace{\varpi_{a}\varpi_{a+1}\cdots\varpi_{b}}_{k} = k! e_k(x_1,x_2,\ldots,x_b) ,
\end{align*} 
where $k=b-a+1$.
\end{proposition}

\begin{proof}
For the case $k=1$ (i.e., $b=a$), the claim is obvious by \eqref{eq: def of varpi i}. 
We assume that $k\ge2$, and we prove the claim by induction on $k$.
By the inductive hypothesis, the left hand side can be computed as
\begin{align*}
&(\varpi_{a}\varpi_{a+1}\cdots\varpi_{b-1})\cdot \varpi_b \\ 
&\hspace{20pt}= (k-1)! e_{k-1}(x_1,x_2,\ldots,x_{b-1}) \cdot \varpi_b \\
&\hspace{20pt}= (k-1)! e_{k-1}(x_1,x_2,\ldots,x_{b-1}) \cdot (x_1+x_2+\cdots+x_{b-1}+x_b) \\
&\hspace{20pt}= (k-1)! \Big( (x_1+\cdots+x_{b-1})e_{k-1}(x_1,\ldots,x_{b-1}) +  e_{k-1}(x_1,\ldots,x_{b-1})x_b \Big).
\end{align*} 
Applying \eqref{eq: fund identity 2} to the last expression, we obtain
\begin{equation}\label{eq: elem 190}
\begin{split}
&(\varpi_{a}\varpi_{a+1}\cdots\varpi_{b-1})\cdot \varpi_b \\
&\hspace{20pt}=
(k-1)! \Big( m_{2,1^{k-2}}(x_1,x_2,\ldots,x_{b-1})+ km_{1^{k}}(x_1,x_2,\ldots,x_{b-1}) \\
&\hspace{250pt} + e_{k-1}(x_1,\ldots,x_{b-1})x_b \Big).
\end{split}
\end{equation} 
For the second summand in the parenthesis of the right hand side, we have
\begin{align*}
km_{1^{k}}(x_1,x_2,\ldots,x_{b-1}) 
&= k e_k(x_1,x_2,\ldots,x_{b-1}) \qquad \text{(by the definition of $m_{1^{k}}$)}\\
&= k e_k(x_1,x_2,\ldots,x_{b}) - k e_{k-1}(x_1,x_2,\ldots,x_{b-1})x_b .
\end{align*}
Applying this to \eqref{eq: elem 190}, we obtain
\begin{equation*}
\begin{split}
&(\varpi_{a}\varpi_{a+1}\cdots\varpi_{b-1})\cdot \varpi_b \\
&\hspace{20pt}=
(k-1)! \Big( m_{2,1^{k-2}}(x_1,x_2,\ldots,x_{b-1}) + k e_k(x_1,x_2,\ldots,x_{b})\\
&\hspace{230pt} - (k-1)e_{k-1}(x_1,\ldots,x_{b-1})x_b \Big) \\
&\hspace{20pt}=
(k-1)! \Big( m_{2,1^{k-2}}(x_1,x_2,\ldots,x_{b-1}) + k e_k(x_1,x_2,\ldots,x_{b})\\
&\hspace{230pt} - (k-1) m_{1^{k-1}}(x_1,\ldots,x_{b-1})x_b \Big) ,
\end{split}
\end{equation*} 
where we used $e_{k-1}(x_1,\ldots,x_{b-1})=m_{1^{k-1}}(x_1,\ldots,x_{b-1})$ again for the last equality. 
In the last expression, the sum of the first and the third summands is equal to $0$ by the case $d=1$ of Lemma~\ref{lem: presentation 20} since $0\le k-2<b-1$.
Thus, the last expression is equal to $k!e_k(x_1,x_2,\ldots,x_b)$.
\end{proof}

\vspace{8pt}

We now obtain the following formula which expresses $\varpi_{J}$ in \eqref{eq: def of varpiJ} as an integer coefficient polynomial in $x_1,x_2,\ldots,x_n$.

\begin{corollary}\label{cor: presentation 20}
For $J\subseteq[n-1]$, we have 
\begin{align*}
\varpi_{J} 
= \prod_{k=1}^m e_{|J_k|}(x_1,x_2,\ldots,x_{\max J_k}) 
\quad \text{in $H^{2|J|}(Pet_n;\Z)$},
\end{align*} 
where $J=J_1\sqcup J_2 \sqcup \cdots \sqcup J_m$ is the decomposition into the connected components.
\end{corollary}

\begin{proof}
The previous proposition implies that
\begin{align*}
 \prod_{i\in J} \varpi_i 
 = \prod_{k=1}^m \Big( \prod_{i\in J_k} \varpi_i  \Big)
 = \prod_{k=1}^m \Big( |J_k|! e_{|J_k|}(x_1,x_2,\ldots,x_{\max J_k}) \Big)
\end{align*}
because each component $J_k$ is of the form $J_k=\{a,a+1,\ldots,b\}$
for some $a,b\in[n-1]$.
Since $H^*(Pet_n;\Z)$ is torsion free by Proposition~\ref{prop: known things 0 Pet}, we obtain the claim by dividing both sides of this equality by $m_J=|J_1|!|J_2|!\cdots |J_m|!$.
\end{proof}

\vspace{5pt}

\begin{example}\label{eg: varpi=e}
{\rm 
Let $n=7$. Then we have
\begin{align*}
 &\varpi_{\{3,4,5\}} = \frac{1}{6}\varpi_3\varpi_4\varpi_5 = e_3(x_1,x_2,x_3,x_4,x_5), \\
 &\varpi_{\{2,4,5\}} = \frac{1}{2}\varpi_2\varpi_4\varpi_5 =  e_1(x_1,x_2)e_2(x_1,x_2,x_3,x_4,x_5).
\end{align*} 
}
\end{example}

\vspace{10pt}

\subsection{A ring presentation of $H^*(\Pet{n};\Z)$}\label{subsec: ring presentation for Pet}
Recall from \eqref{eq: def of x i} that we have
\begin{align*}
 x_i = c_1((E_i/E_{i-1})^*|_{\Pet{n}}) 
 \in H^2(\Pet{n};\Z)
\end{align*}
and that it is the image of $\tau_i=c_1((E_i/E_{i-1})^*)\in H^2(Fl_n;\Z)$ under the restriction map 
\begin{align*}
i^*\colon H^*(Fl_n;\Z)\rightarrow H^*(\Pet{n};\Z).
\end{align*}
The ring $H^*(Fl_n;\Z)$ is generated by $\tau_1,\tau_2,\ldots,\tau_n$ as we saw in section~\ref{subsec: coh of flag}, and the map $i^*$ is surjective by \cite{Insko} (see Corollary~\ref{corollary: known things 3}).
Thus, we have a surjective ring homomorphism
\begin{align*}
 \RP \colon  \Z[y_1,y_2,\ldots,y_n] \rightarrow H^*(\Pet{n};\Z)
\end{align*}
which sends $y_i$ to $x_i$ $(1\le i\le n)$, where $\Z[y_1,y_2,\ldots,y_n]$ is the polynomial ring over $\Z$ with indeterminates $y_1,y_2,\ldots,y_n$. 
We regard this polynomial ring as a graded ring with $\deg y_i=2$ for $1\le i\le n$.
By construction, this map factors $H^*(Fl_n;\Z)$, and hence it maps $e_k(y_1,y_2,\ldots,y_n)$ to $0$ in $H^*(\Pet{n};\Z)$ for $1\le k\le n$ (see section~\ref{subsec: coh of flag}).

To give a ring presentation of $H^*(\Pet{n};\Z)$, we introduce the following homogeneous ideals of $\Z[y_1,y_2,\ldots,y_n]$: 
\begin{equation}
\begin{split}\label{eq: def of I and I'}
 &I \coloneqq (e_k(y_1,y_2,\ldots,y_n) \mid 1\le k\le n ), \\
 &I' \coloneqq ((y_i-y_{i+1})e_k(y_1,\ldots,y_i) \mid 1\le i\le n-1,\ 1\le k\le \min\{i,n-i\}). 
 \end{split}
\end{equation}

\vspace{10pt}

\begin{example}\label{eg: n=4}
{\rm
Let $n=4$. The ideal $I$ of $\Z[y_1,y_2,y_3,y_4]$ is generated by 
\begin{align}\label{eq: generator list 1}
 e_1(y_1,y_2,y_3,y_4), \ e_2(y_1,y_2,y_3,y_4), \ e_3(y_1,y_2,y_3,y_4), \ e_4(y_1,y_2,y_3,y_4), 
\end{align}
and the ideal $I'$ is generated by 
\begin{align}\label{eq: generator list 2}
 (y_1-y_2)y_1, \
 (y_2-y_3)(y_1+y_2), \ (y_2-y_3)y_1y_2, \ 
 (y_3-y_4)(y_1+y_2+y_3). 
\end{align}}
\end{example}

\vspace{10pt}

From what we saw above, it is clear that the map $\RP$ sends the ideal $I$ to $0$ in $H^*(\Pet{n};\Z)$.
It follows that $\RP$ also sends $I'$ to $0$.
To see that,
it suffices to show that
\begin{align}\label{eq: fundamental relation in Pet}
 (x_i-x_{i+1})e_k(x_1,\ldots,x_i) =0 
 \quad \text{$(1\le k\le i\le n-1)$}
\end{align}
in $H^*(\Pet{n};\Z)$, where we note that the range of $k$ is larger than that of in \eqref{eq: def of I and I'}.
For this purpose, let $1\le k\le i\le n-1$.
By Proposition~\ref{prop: presentation 10}, we have 
\begin{align*}
 k! e_k(x_1,\ldots,x_i) 
 = \varpi_{i-k+1}\varpi_{i-k+2}\cdots\varpi_{i} .
\end{align*}
Recalling that $x_i-x_{i+1}=\alpha_i$, we obtain
\begin{align*}
 k! (x_i-x_{i+1}) e_k(x_1,\ldots,x_i) 
 = \alpha_i \varpi_{i-k+1}\varpi_{i-k+2}\cdots\varpi_{i} .
\end{align*}
The right hand side of this equality is $0$
by \eqref{eq: alphapi=0} so that we obtain
\begin{align*}
 k! (x_i-x_{i+1})e_k(x_1,\ldots,x_i) = 0.
\end{align*}
Since $H^*(\Pet{n};\Z)$ is torsion free by Proposition~\ref{prop: known things 0 Pet}, this implies the equality \eqref{eq: fundamental relation in Pet}. 
Hence, $\RP$ sends $I'$ to $0$, as we claimed above.

\begin{remark}
{\rm
Geometric meaning of the relation \eqref{eq: fundamental relation in Pet} can be explained as follows. 
By Corollary~\ref{cor: presentation 20}, it can be expressed as 
\begin{align*}
 \alpha_{i} \cdot \varpi_J =0 ,
\end{align*}
where we set $J=\{i-k+1,i-k+2,\ldots,i\}$. 
In \cite{AHKZ}, two kinds of closed subsets $X_J(=\Pet{J})$ and $\Omega_J$ in $\Pet{n}$ are introduced, and this equality can be explained from the corresponding geometric equality 
\begin{align*}
 X_{\{i\}} \cap \Omega_{J} = \emptyset 
\end{align*}
by an argument similar to that in the proof of \cite[Lemma~4.7]{AHKZ}.
See \cite[Sect.\ 3 and 4]{AHKZ} for details.
}
\end{remark}

\vspace{10pt}

Since the map $\phi$ sends both of $I$ and $I'$ to $0$, it induces a surjective ring homomorphism
\begin{align*}
 \overline{\RP} \colon  \Z[y_1,y_2,\ldots,y_n]/(I+I') \rightarrow H^*(\Pet{n};\Z)
\end{align*}
which sends $y_i$ to $x_i$ $(1\le i\le n)$.
Here, we use the same symbol $y_i$ for its image in the quotient $\Z[y_1,y_2,\ldots,y_n]/(I+I')$ by abusing notation. We adopt this notation in the rest of this paper.

The next claim gives Theorem~\ref{thm: B} in Section~\ref{sec: intro} which describes
the ring structure of $H^*(\Pet{n};\Z)$.

\begin{theorem}\label{thm: main theorem 2}
The induced homomorphism
\begin{align*}
 \overline{\RP} \colon  \Z[y_1,y_2,\ldots,y_n]/(I+I') \rightarrow H^*(\Pet{n};\Z)
\end{align*}
sending $y_i$ to $x_i= c_1((E_i/E_{i-1})^*|_{\Pet{n}}) $ $(1\le i\le n)$
is an isomorphism as graded rings, where $I$ and $I'$ are the ideals of $\Z[y_1,y_2,\ldots,y_n]$ defined in \eqref{eq: def of I and I'}.
\end{theorem}

\vspace{10pt}

The rest of this paper is devoted for the proof of Theorem~\ref{thm: main theorem 2}.

\begin{remark}
The relation \eqref{eq: fundamental relation in Pet} for $k=1$ takes the form
\begin{align*}
 \alpha_i\cdot \varpi_i =
 (x_i-x_{i+1})(x_1+\cdots+x_i) = 0 \quad \text{in $H^*(\Pet{n};\Z)$}  
\end{align*}
which appears as the fundamental relations of the presentation of the cohomology ring $H^*(\Pet{n};\C)$ in \cite[Corollary~3.4]{fu-ha-ma} and \cite[Theorem 4.1]{ha-ho-ma} $($cf.\ \cite[Remark~4.8]{AHKZ}$)$.
\end{remark}

\vspace{10pt}

\begin{remark}
Let $n=4$. If we remove $(y_2-y_3)y_1y_2$ from the list \eqref{eq: generator list 2} of the generators of $I'$, then the quotient ring $\Z[y_1,y_2,y_3,y_4]/(I+I')$ is not isomorphic to $H^*(\Pet{4};\Z)$ since $(y_2-y_3)y_1y_2$ gives a non-zero $2$-torsion element of the quotient ring. This can be verified by a computer assisted calculation. Similarly, if we remove $e_k(y_1,y_2,y_3,y_4)$ for some $1\le k\le 4$ from the list \eqref{eq: generator list 1} of the generators of $I$, then the quotient ring $\Z[y_1,y_2,y_3,y_4]/(I+I')$ is not isomorphic to $H^*(\Pet{4};\Z)$. 
\end{remark}

\vspace{10pt}

Set
\begin{align*}
 M \coloneqq \Z[y_1,y_2,\ldots,y_n]/(I+I').
\end{align*}
If we can construct a subset $\{\pi_J \mid J\subseteq[n-1]\}$ of $M$ which generates $M$ as a $\Z$-module and satisfies $\overline{\RP}(\pi_J)=\varpi_J$ for $J\subseteq[n-1]$, then it follows that the map $\overline{\RP}$ is an isomorphism since $\varpi_J$ for $J\subseteq[n-1]$ form a $\Z$-basis of $H^*(\Pet{n};\Z)$ by Theorem~\ref{thm: main thm of AHKZ}.

Motivated by Corollary~\ref{cor: presentation 20}, 
we define $\pi_J \in M$ for each $J\subseteq[n-1]$ as follows.
For a subset $J\subseteq [n-1]$ having the decomposition $J=J_1\sqcup \cdots \sqcup J_m$ into the connected components (see section~\ref{subsec: comb on J}), we set
\begin{align}\label{eq: def of algebraic piJ}
 \pi_J 
 \coloneqq 
 \prod_{k=1}^m e_{|J_k|}(y_1,y_2,\ldots,y_{\max J_k}) \in M,
\end{align}
where we take the convention
\begin{align}\label{eq: def of algebraic piJ 2}
 \pi_{\emptyset}=1\in M.
\end{align}
We also define $\pi_J$ for all $J\subseteq \Z$ by taking the convention 
\begin{align}\label{eq: def of algebraic piJ 3}
 \pi_J = 0 \quad \text{unless $J\subseteq [n-1]$}.
\end{align}

\begin{example}
{\rm 
Let $n=7$. Then we have
\begin{align*}
 &\pi_{\{3,4,5\}} = e_3(y_1,y_2,y_3,y_4,y_5), \\
 &\pi_{\{2,4,5\}} = e_1(y_1,y_2)e_2(y_1,y_2,y_3,y_4,y_5)
\end{align*} 
in $M$ (cf.\ Example~\ref{eg: varpi=e}). We also have $\pi_{\{0,2\}}=0$ and $\pi_{\{4,5,7\}} = 0$ by \eqref{eq: def of algebraic piJ 3}.}
\end{example}

\vspace{10pt}

Recall that we have $\overline{\RP}(y_i)=x_i$ for $1\le i\le n$ by definition. Hence, it is clear that we have $\overline{\RP}(\pi_J)=\varpi_J$ for $J\subseteq[n-1]$ by Corollary~\ref{cor: presentation 20}.
Thus, to prove Theorem~\ref{thm: main theorem 2}, it is enough to show the following claim as we discussed above. 

\begin{proposition}\label{prop: module generator in y}
The $\Z$-module $M=\Z[y_1,y_2,\ldots,y_n]/(I+I')$ is generated by 
the subset $\{ \pi_J\mid J\subseteq[n-1]\}$. 
\end{proposition}

We prove this in the next subsection.

\vspace{10pt}

\subsection{A Proof of Proposition \ref{prop: module generator in y}}
Before giving a proof of Proposition \ref{prop: module generator in y}, we first  establish some basic properties of $\pi_J$ for $J\subseteq[n-1]$. 
We begin with the following identity in $M$: 
\begin{equation}
\begin{split}\label{eq: elem 95}
 e_k(y_1,y_2,\ldots,y_n) = 
 \sum_{\substack{0\le p\le i,\ 0\le q\le n-i \\ p+q=k}} 
 e_{p}(y_1,y_2,\ldots,y_i) e_{q}(y_{i+1},y_{i+2},\ldots,y_n) 
\end{split}
\end{equation} 
for $1\le k\le i\le n$, where we take the convention $e_{0}=1$.
One may obtain this identity by decomposing the index set $[n]$ of monomials in the left hand side into two parts: $[n]=\{1,2,\ldots,i\}\sqcup\{i+1,i+2,\ldots,n\}$.

For $1\le i\le n-1$, recall from the definition of $M$ that we have
\begin{align*}
(y_i-y_{i+1}) e_k(y_1,\ldots,y_i) = 0
\qquad (1\le k\le \min\{i,n-i\}).
\end{align*} 
The following claim means that the same equalities hold for a wider range of $k$.

\begin{lemma}\label{lem: presentation 60}
For $1\le i\le n-1$, we have
\begin{align*}
(y_i-y_{i+1}) e_k(y_1,\ldots,y_i) = 0 
\qquad (1\le k\le i)
\end{align*} 
in $M$.
\end{lemma}

\begin{proof}
If $1\le k\le n-i$, then we have $k\le\min\{i,n-i\}$ (since $k\le i$), and the claim holds by the definition of $M$ as we saw above.

If $n-i<k\le i$, then we prove the claim by induction on $k$, where we note that  we already verified the claim for $1\le k\le n-i$. Hence, we assume by induction that the claim holds when $k\le \ell$ for some positive integer $\ell$ satisfying $n-i\le \ell<i$, and 
we prove the claim when $k=\ell+1$.
Since we have $e_k(y_1,\ldots,y_n)=0$ by the definition $M$, we know that
\begin{align*}
&(y_i-y_{i+1}) e_k(y_1,\ldots,y_n) = 0.
\end{align*} 
By \eqref{eq: elem 95}, this equality can be written as
\begin{align*}
\sum_{\substack{0\le p\le i,\ 0\le q\le n-i \\ p+q=k}} 
(y_i-y_{i+1}) e_{p}(y_1,\ldots,y_i) e_{q}(y_{i+1},\ldots,y_n)=0.
\end{align*} 
Because of the condition $1\le k\le i$, it follows that the summand for $q=0$ (i.e., $p=k$) appears in the left hand side of this equality. Thus, we can separate it to obtain
\begin{align*}
(y_i-y_{i+1}) e_k(y_1,\ldots,y_i)
+ 
\sum_{\substack{0\le p\le i,\ 1\le q\le n-i \\ p+q=k}} 
(y_i-y_{i+1}) e_{p}(y_1,\ldots,y_i) e_{q}(y_{i+1},\ldots,y_n) =0.
\end{align*} 
In the second summand, we have $p=k-q>(n-i)-q\ge0$
since we are considering the case $n-i<k\le i$.
This implies that $p\ge1$ in the second summand. 
Noticing that $1\le p < k$, the inductive hypothesis implies that
the second summand vanishes. Thus, we obtain
\begin{align*}
(y_i-y_{i+1}) e_k(y_1,\ldots,y_i) =0.
\end{align*} 
\end{proof}

\vspace{10pt}

For non-negative integers $a$ and $b$, we use the notation
\begin{align*}
 [a,b] \coloneqq \{c\in\Z \mid a\le c\le b\}
\end{align*}
in what follows. For example, we have $[2,5]=\{2,3,4,5\}$,  $[3,2]=\emptyset$, and $[0,1]=\{0,1\}$ so that
\begin{align*}
 \pi_{[2,5]}= e_4(y_1,y_2,\ldots,y_5), \quad \pi_{[3,2]} = 1, \quad \text{and} \quad 
 \pi_{[0,1]} = 0
\end{align*}
by \eqref{eq: def of algebraic piJ}, \eqref{eq: def of algebraic piJ 2}, and \eqref{eq: def of algebraic piJ 3}, respectively.
Noticing that we have
\begin{align*}
e_{k}(y_1,\ldots,y_{b}) = e_{k}(y_1,\ldots,y_{b-1}) + e_{k-1}(y_1,\ldots,y_{b-1})y_{b} \quad (1\le k\le b\le n-1), 
\end{align*} 
it follows that
\begin{align}\label{eq: elem 50 2}
\pi_{[a,b]}= \pi_{[a-1,b-1]} + \pi_{[a,b-1]}y_{b}\quad (1\le a\le b\le n-1)
\end{align} 
since we have $\pi_{[a,b]}=e_{b-a+1}(y_1,y_2,\ldots,y_b)$.
When $a=b$, we obtain
\begin{align*}
\pi_{[a,a]}= \pi_{[a-1,a-1]} + y_{a}\quad (1\le a\le n-1)
\end{align*} 
since $\pi_{[a,a-1]}=\pi_{\emptyset}=1$.
The next claim generalizes Lemma~\ref{lem: presentation 60}.

\vspace{10pt}

\begin{lemma}\label{lemma: other alpha i}
For  $1\le a\le b\le n-1$, we have
\begin{align}\label{eq: joint lemma 2 in y by pi 4}
 (y_i-y_{i+1}) \cdot \pi_{[a,b]} = 0
 \qquad (a\le  i\le  b)
\end{align} 
in $M$.
\end{lemma}

\begin{proof}
When $i=b$, the claim follows by Lemma~\ref{lem: presentation 60} since we have 
\begin{align*}
 (y_i-y_{i+1})\pi_{[a,i]}=(y_i-y_{i+1})e_{i-a+1}(y_1,\ldots,y_{i})=0
\end{align*} 
in this case.
We prove the claim by  induction on $b-i\ge0$.
Let $\ell(<n-1)$ be a non-negative integer, and assume by induction that \eqref{eq: joint lemma 2 in y by pi 4} holds for $b-i=\ell$.
We prove that \eqref{eq: joint lemma 2 in y by pi 4} holds for $b-i=\ell+1(\ge1)$.
By \eqref{eq: elem 50 2}, we have
\begin{align}\label{eq: elem 300}
 (y_i-y_{i+1}) \cdot \pi_{[a,b]} 
 &=
 (y_i-y_{i+1}) \pi_{[a-1,b-1]} + (y_i-y_{i+1}) \pi_{[a,b-1]} y_b.
\end{align} 
We compute the right hand side by taking cases.
If $a=1$, then the first summand is equal to zero by \eqref{eq: def of algebraic piJ 3}, and the second summand is also equal to zero by the inductive hypothesis since $b-1\ge i$. If $a>1$, both summands are equal to zero by the inductive hypothesis. Thus, in either case, the right hand side of \eqref{eq: elem 300} is equal to zero so that we obtain $(y_i-y_{i+1}) \cdot \pi_{[a,b]} =0$.
\end{proof}

\vspace{10pt}

In particular, for $1\le a\le b\le n-1$, we obtain
\begin{align}\label{eq: joint lemma 2 in y by pi 5}
 y_i \pi_{[a,b]} = y_{b+1} \pi_{[a,b]}
 \qquad (a\le  i\le  b)
\end{align} 
in $M$ by applying \eqref{eq: joint lemma 2 in y by pi 4} repeatedly.

To state the next claim, let us recall a basic property of $\varpi_J$ in the cohomology $H^*(\Pet{n};\Z)$: it is clear from the definition \eqref{eq: def of varpiJ} that, for $1\le a\le i< b\le n-1$, we have
\begin{align*}
 \varpi_{[a,i]}\cdot \varpi_{[i+1,b]}
 = \frac{1}{(i-a+1)!} \cdot \frac{1}{(b-i)!} \cdot \varpi_{a}\varpi_{a+1}\cdots\varpi_{b}
 = \binom{b-a+1}{i-a+1} \varpi_{[a,b]}
\end{align*} 
in $H^*(\Pet{n};\Z)$, where $\binom{b-a+1}{i-a+1}$ is a binomial coefficient.
As the following claim shows, the analogous equalities hold in $M$ as well.

\vspace{10pt}

\begin{proposition}\label{prop: joint lemma 2 in y}
For $1\le a\le i<  b \le n-1$, we have
\begin{align}\label{eq: joint lemma 2 in y by pi 3}
 \pi_{[a,i]}\cdot \pi_{[i+1,b]}
 = \binom{b-a+1}{i-a+1} \pi_{[a,b]}
\end{align} 
in $M$.
\end{proposition}

\begin{proof}
We prove the claim by induction on $b\ge2$.
When $b=2$, we have $a=i=1$ so that the claim is
\begin{align*}
 \pi_{[1,1]}\cdot \pi_{[2,2]}
 = 2 \pi_{[1,2]}.
\end{align*}
The left hand  side can be computed as
\begin{align*}
 \pi_{[1,1]}\cdot \pi_{[2,2]}
 = \pi_{[1,1]}\cdot \pi_{\{2\}}
 = \pi_{[1,1]}\cdot(y_1+y_2)
 = \pi_{[1,1]}\cdot 2y_2
 = 2\pi_{[1,2]},
\end{align*}
where the third equality follows from \eqref{eq: joint lemma 2 in y by pi 5}, and the fourth equality follows from \eqref{eq: elem 50 2} with $\pi_{[0,1]}=0$. Thus, we obtain the claim for $b=2$.

Let $2\le \ell< n-1$ be a positive integer, and assume by induction that the claim \eqref{eq: joint lemma 2 in y by pi 3} holds for $b=\ell$.
We prove the claim \eqref{eq: joint lemma 2 in y by pi 3} for $b=\ell+1$.
The left hand side of \eqref{eq: joint lemma 2 in y by pi 3} can be written as
\begin{align*}
\pi_{[a,i]}\cdot \pi_{[i+1,b]} 
&= \pi_{[a,i]} \pi_{[i,b-1]} + \pi_{[a,i]} \pi_{[i+1,b-1]} y_{b} 
\end{align*}
by \eqref{eq: elem 50 2}. Applying \eqref{eq: elem 50 2} again to $\pi_{[a,i]}$ in the first summand of the right hand side, we obtain
\begin{align}\label{eq: elem 50 5}
\pi_{[a,i]}\cdot \pi_{[i+1,b]} 
&= \pi_{[a-1,i-1]} \pi_{[i,b-1]}+\pi_{[a,i-1]}\pi_{[i,b-1]}y_i +\pi_{[a,i]} \pi_{[i+1,b-1]}y_{b} .
\end{align}
We now compute each summands of the right hand side separately. For the first summand, we have
\begin{align*}
 \pi_{[a-1,i-1]} \pi_{[i,b-1]}= \binom{b-a+1}{i-a+1}\pi_{[a-1,b-1]}
\end{align*}
which we prove by taking cases as follows.
If $a=1$, the claim is obvious since both sides are equal to $0$ by \eqref{eq: def of algebraic piJ 3}.
If $a>1$, then the claim follows by the inductive hypothesis.
For the second summand of the right hand side of \eqref{eq: elem 50 5}, we have
\begin{align*}
\pi_{[a,i-1]}\pi_{[i,b-1]}y_i = \binom{b-a}{i-a}\pi_{[a,b-1]}y_i 
\end{align*}
which we prove by taking cases as follows.
If $a=i$, the claim is obvious since both sides are equal to $\pi_{[i,b-1]}y_i$ by \eqref{eq: def of algebraic piJ 2}.
If $a<i$, then the claim follows by the inductive hypothesis.
For the third summand of the right hand side of \eqref{eq: elem 50 5}, we have
\begin{align*}
\pi_{[a,i]} \pi_{[i+1,b-1]}y_{b}
= \binom{b-a}{i-a+1} \pi_{[a,b-1]}y_{b} 
\end{align*}
which we prove by taking cases as follows.
If $b=i+1$, the claim is obvious since both sides are equal to $\pi_{[a,i]}y_b$ by \eqref{eq: def of algebraic piJ 2}.
If $b>i+1$, then the claim follows by the inductive hypothesis.

Combining the above computations of the summands of the right hand side of \eqref{eq: elem 50 5}, we obtain
\begin{align*}
\pi_{[a,i]}\cdot \pi_{[i+1,b]} = \binom{b-a+1}{i-a+1}\pi_{[a-1,b-1]}+\binom{b-a}{i-a}\pi_{[a,b-1]}y_i +\binom{b-a}{i-a+1} \pi_{[a,b-1]}y_{b} .
\end{align*}
By applying \eqref{eq: joint lemma 2 in y by pi 5} to the second summand of the right hand side, it follows that
\begin{align*}
\pi_{[a,i]}\cdot \pi_{[i+1,b]} 
&= \binom{b-a+1}{i-a+1}\pi_{[a-1,b-1]}+\binom{b-a}{i-a}\pi_{[a,b-1]}y_b +\binom{b-a}{i-a+1} \pi_{[a,b-1]}y_{b} \\
&= \binom{b-a+1}{i-a+1}\pi_{[a-1,b-1]}+\binom{b-a+1}{i-a+1}\pi_{[a,b-1]}y_{b} \\
&= \binom{b-a+1}{i-a+1}\pi_{[a,b]}, 
\end{align*}
where we used \eqref{eq: elem 50 2} for the last equality. This completes the proof.
\end{proof}

\vspace{40pt}

For simplicity, we write
\begin{align}\label{eq: simplicity of index}
 \pi_i \coloneqq \pi_{\{i\}} = y_1+y_2+\cdots+y_i \quad (1\le i\le n-1)
\end{align}
(cf.\ \eqref{eq: def of varpi i}).
As for the previous proposition, the following is an analogue of \cite[Lemma~5.1]{AHKZ} which is a claim for $\varpi_J$ in the cohomology $H^*(\Pet{n};\Z)$.

\begin{proposition}\label{prop: joint lemma in y}
For $1\le a\le i\le b \le n-1$, we have
\begin{align}\label{eq: joint lemma in y by pi 2}
&\pi_i \cdot \pi_{[a,b]}
= (b-i+1) \pi_{[a-1,b]} + (i-a+1) \pi_{[a,b+1]}
\end{align} 
in $M$ with the convention $\pi_{[0,b]}=\pi_{[a,n]}=0$ $($See \eqref{eq: def of algebraic piJ 3}$)$.
\end{proposition}

Before giving a proof, we recall that
the identity \eqref{eq: fund identity 2} was obtained simply by rearranging terms, without using any non-trivial relations between the cohomology classes $x_1,x_2,\ldots,x_n$. Thus, it follows that the same identity holds for $y_1,y_2,\ldots,y_n$ in $M$ as well:
\begin{equation}\label{eq: elem 200}
\begin{split}
&(y_1+y_2+\cdots+y_b) \cdot e_k(y_1,y_2,\ldots,y_b) \\
&\hspace{30pt}= m_{2,1^{k-1}}(y_1,y_2,\ldots,y_{b}) + (k+1) m_{1^{k+1}}(y_1,y_2,\ldots,y_b) \\
&\hspace{30pt}= m_{2,1^{k-1}}(y_1,y_2,\ldots,y_{b}) + (k+1) e_{k+1}(y_1,y_2,\ldots,y_b) 
\end{split}
\end{equation}
for $1\le k\le b\le n$.
We also recall that the claim of Lemma~\ref{lem: presentation 20} was derived only by the relations $\alpha_i\varpi_i=0$ $(1\le i\le n-1)$\footnote{In section~\ref{subsec: basis of Pet}, the condition that $H^*(\Pet{n};\Z)$ is torsion free was used only  in the proof of Corollary~\ref{cor: presentation 20}.}. 
By the definitions of $M$ and $I'$, we have the corresponding relations $(y_i-y_{i+1})\pi_i=0$ $(1\le i\le n-1)$ in $M$.
Thus, the claim of Lemma~\ref{lem: presentation 20} holds for $y_1,\ldots,y_n$ as well. 
Its claim for $d=1$ gives us that
\begin{align*}
m_{2,1^{k-1}}(y_1,\ldots,y_{b})=km_{1^{k}}(y_1,\ldots,y_{b}) y_{b+1}=k e_{k}(y_1,\ldots,y_{b}) y_{b+1}.
\end{align*}
Applying this to the last expression in \eqref{eq: elem 200}, we obtain  that
\begin{align*}
&(y_1+y_2+\cdots+y_b) \cdot e_k(y_1,y_2,\ldots,y_b)
\\
&\hspace{30pt}=
k e_{k}(y_1,\ldots,y_{b}) y_{b+1} + (k+1) e_{k+1}(y_1,y_2,\ldots,y_b) 
\end{align*}
for $1\le k\le b\le n$. This can be expressed as
\begin{equation}\label{eq: elem 200 3}
\pi_b \cdot \pi_{[a,b]}
= (b-a+1) \pi_{[a,b]} y_{b+1} + (b-a+2) \pi_{[a-1,b]}
\qquad (1\le a\le b\le n).
\end{equation}
\vspace{0pt}

\begin{proof}[Proof of Proposition~$\ref{prop: joint lemma in y}$]
We first prove the equality for the case $i=b$. In this case, the claim is 
\begin{align}\label{eq: case i=b}
\pi_b \cdot \pi_{[a,b]} = \pi_{[a-1,b]} + (b-a+1)\pi_{[a,b+1]}.
\end{align} 
The left hand side is equal to 
\begin{align*}
\pi_b \cdot \pi_{[a,b]}
= (b-a+1) \pi_{[a,b]} y_{b+1} + (b-a+2) \pi_{[a-1,b]}
\end{align*}
by \eqref{eq: elem 200 3}.
Thus, we obtain
\begin{equation}
\begin{split}\label{eq: taking cases 5}
\pi_b \cdot \pi_{[a,b]}
&= (b-a+1) \pi_{[a,b]} y_{b+1} + (b-a+2) \pi_{[a-1,b]}\\
&= \pi_{[a-1,b]} + (b-a+1)\Big(\pi_{[a-1,b]} + \pi_{[a,b]} y_{b+1} \Big).
\end{split}
\end{equation} 
In the parenthesis of the last expression, we have
\begin{align}\label{eq: taking cases 6}
 \pi_{[a-1,b]} + \pi_{[a,b]} y_{b+1}=\pi_{[a,b+1]}
\end{align}
which we prove by taking cases.
If $b<n-1$, it is obvious by \eqref{eq: elem 50 2}. If $b=n-1$, it can be shown as
\begin{align*}
 \pi_{[a-1,b]} + \pi_{[a,b]} y_{b+1}
 &= 
 \pi_{[a-1,n-1]} + \pi_{[a,n-1]}y_{n}\\
 &= e_{n-a+1}(y_1,y_2,\ldots,y_{n-1}) + e_{n-a}(y_1,y_2,\ldots,y_{n-1})y_n \\
 &=  e_{n-a+1}(y_1,y_2,\ldots,y_{n}) \\
 &= 0 \qquad (\text{by the definitions of $M$ and $I$}) \\
 &= \pi_{[a,b+1]} \qquad (\text{by $b+1=n$ and \eqref{eq: def of algebraic piJ 3}}) 
\end{align*} 
in this case as  well.
Thus, we obtain from \eqref{eq: taking cases 5} and \eqref{eq: taking cases 6} that
\begin{align*}
\pi_b \cdot \pi_{[a,b]} = \pi_{[a-1,b]}+ (b-a+1)\pi_{[a,b+1]} .
\end{align*} 
This verifies \eqref{eq: case i=b} which is the desired claim \eqref{eq: joint lemma in y by pi 2} for the case $i=b$. 

Now we prove the claim \eqref{eq: joint lemma in y by pi 2} by induction on $b-i\ge0$. 
Let $\ell(<n-1)$ be a non-negative integer, and assume by induction that \eqref{eq: joint lemma in y by pi 2} holds when $b-i=\ell$. We prove \eqref{eq: joint lemma in y by pi 2} for the case $b-i=\ell+1(\ge1)$.
In this case, the left hand side of \eqref{eq: joint lemma in y by pi 2} can be computed as
\begin{align}\label{eq: induction taking cases}
\pi_i \cdot \pi_{[a,b]}
= \pi_i \pi_{[a-1,b-1]} + \pi_i \pi_{[a,b-1]}y_{b}
\end{align}
by \eqref{eq: elem 50 2}. We compute each summands of the right hand side separately.
For the first summand, we have
\begin{align*}
\pi_i \pi_{[a-1,b-1]}
= (b-i) \pi_{[a-2,b-1]} + (i-a+2) \pi_{[a-1,b]} 
\end{align*}
which we prove by taking cases as follows.
If $a=1$, then both sides are equal to $0$ by \eqref{eq: def of algebraic piJ 3}. If $a>1$, then the claim  follows by the inductive hypothesis since $b-1\ge i$.
For the second summand in the right hand side of \eqref{eq: induction taking cases},  we  have
\begin{align*}
\pi_i \pi_{[a,b-1]}y_{b}
= \Big( (b-i) \pi_{[a-1,b-1]} + (i-a+1) \pi_{[a,b]} \Big) y_{b}
\end{align*}
by the inductive hypothesis since $b-1\ge i$.
Combining the computations of the summands of the right hand side of \eqref{eq: induction taking cases}, we obtain
\begin{equation}
\begin{split}\label{eq: taking cases 2}
\pi_i \cdot \pi_{[a,b]}
&= \Big( (b-i) \pi_{[a-2,b-1]} + (i-a+2) \pi_{[a-1,b]} \Big) \\
&\hspace{50pt} +\Big((b-i) \pi_{[a-1,b-1]} + (i-a+1) \pi_{[a,b]} \Big)y_{b}  \\
&= (b-i)\Big( \pi_{[a-2,b-1]} + \pi_{[a-1,b-1]} y_{b}\Big) \\
&\hspace{50pt} +\pi_{[a-1,b]} \ +\ (i-a+1) \Big( \pi_{[a-1,b]} + \pi_{[a,b]}y_{b} \Big) \\
&= (b-i)\Big( \pi_{[a-2,b-1]} + \pi_{[a-1,b-1]} y_{b}\Big) \\
&\hspace{50pt} +\pi_{[a-1,b]} \ +\ (i-a+1) \Big( \pi_{[a-1,b]} + \pi_{[a,b]}y_{b+1} \Big),
\end{split}
\end{equation}
where we used \eqref{eq: joint lemma 2 in y by pi 5} to $\pi_{[a,b]}y_{b}$ for the last equality.
We now compute the first and  the third summands of the last expression separately. For the first summand, we have
\begin{align}\label{eq: taking cases 3}
 (b-i)\Big( \pi_{[a-2,b-1]} + \pi_{[a-1,b-1]} y_{b}\Big)  = (b-i) \pi_{[a-1,b]}
\end{align}
which we prove by taking cases as follows.
If $a>1$, then the claim is obvious by \eqref{eq: elem 50 2}.
If $a=1$, the claim follows since both sides are equal to $0$ by \eqref{eq: def of algebraic piJ 3}.
For the third summand of the last expression of \eqref{eq: taking cases 2}, we have
\begin{align}\label{eq: taking cases 4}
 (i-a+1) \Big( \pi_{[a-1,b]} + \pi_{[a,b]}y_{b+1} \Big)  = (i-a+1) \pi_{[a,b+1]}
\end{align}
by \eqref{eq: taking cases 6}.
Applying \eqref{eq: taking cases 3} and \eqref{eq: taking cases 4} to the last expression of \eqref{eq: taking cases 2}, we obtain
\begin{align*}
\pi_i \cdot \pi_{[a,b]}
= (b-i+1) \pi_{[a-1,b]} +(i-a+1) \pi_{[a,b+1]}.
\end{align*}
This verifies \eqref{eq: joint lemma in y by pi 2} for the case $b-i=\ell+1$, as desired.
\end{proof}

\vspace{10pt}

We now prove Proposition~\ref{prop: module generator in y} which completes the proof of Theorem~\ref{thm: main theorem 2} as we discussed in section~\ref{subsec: ring presentation for Pet}.

\begin{proof}[Proof of Proposition~$\ref{prop: module generator in y}$]
Let $f\in M$ be a polynomial in $y_1,y_2,\ldots,y_n$.
We prove that $f$ can be written as a linear combination of $\pi_J$ for $J\subseteq[n-1]$. 
Recalling that $\pi_i = y_1+y_2+\cdots+y_i$, we have
\begin{align*}
 &y_i 
 = \pi_i - \pi_{i-1} 
 \quad (1\le i\le n-1), \\
 &y_n= -\pi_{n-1}
\end{align*}
with the convention $\pi_0=0$, 
where the second equality follows since we have an equality 
$y_1+y_2+\cdots+y_n=e_1(y_1,y_2,\ldots,y_n)=0$ in $M$.
Hence, we can express $f$ as a polynomial in 
\begin{align*}
 \pi_1 , \pi_2, \ldots,\pi_{n-1}.
\end{align*}
For our purpose, we may assume that $f$ is a monomial in these variables with coefficient $1$ without loss of generality. Namely, we have
\begin{align}\label{eq: monomial in truncated e1}
 f = \pi_{i_1}\pi_{i_2}\cdots \pi_{i_d}
\end{align}
for some $d\ge1$ and $1\le i_1\le i_2 \le \cdots \le i_d< n$.
We prove that this monomial can be expressed as a linear combination of $\pi_J$ for $J\subseteq[n-1]$.

First suppose that $d \le n-1$.
We prove the claim by induction on $d\ge1$.
When $d=1$, the claim is obvious since $\pi_{i_1}=\pi_{J}$ with $J=\{i_1\}$ by \eqref{eq: simplicity of index}.
Assume by induction that 
\begin{align*}
 \pi_{i_2}\cdots \pi_{i_d}= \sum_{J\subseteq[n-1]} c_{J} \pi_J 
\end{align*}
for some integers $c_J\ (J\subseteq[n-1])$. 
Then we have
\begin{align}\label{eq: final computations 20}
 f 
 = \pi_{i_1} \sum_{J\subseteq[n-1]} c_{J} \pi_J 
 = \sum_{J\subseteq[n-1]} c_{J} \pi_{i_1} \pi_J  .
\end{align}
It suffices to show that each product $\pi_{i_1} \pi_J$ is expanded as a linear combination of $\pi_L$ for $L\subseteq[n-1]$.
For this, we take the decomposition $J=J_1\sqcup \cdots \sqcup J_m$ into the connected components.
By \eqref{eq: def of algebraic piJ}, we have
\begin{align*}
 \pi_{i_1}\pi_J=\pi_{i_1} \cdot \pi_{J_1}\cdots \pi_{J_m}
\end{align*}
since each $J_k$ consists of a single connected component $(1\le k\le m)$.
Set $J' \coloneqq J\cup\{i_1\}$.
We consider the case for $J'\supsetneq J$ and the case for $J'=J$ separately.

If $J'\supsetneq J$, then we have $i_1\notin J$.
We denote by $J'(i_1)\subseteq J'$ the connected component of $J'$ containing $i_1$.
We have the following three cases:
\begin{itemize}
 \item[(1)] $J'(i_1)=\{i_1\}$, 
 \item[(2)] $J'(i_1)=J_k\cup\{i_1\}$ for some $1\le k\le m$, 
 \item[(3)] $J'(i_1)=J_k\cup\{i_1\}\cup J_{k+1}$ for some $1\le k< m$.
\end{itemize}
In the case (1), we have 
\begin{align*}
 \pi_{i_1}\pi_J = \pi_{i_1} \cdot \pi_{J_1}\cdots \pi_{J_m} = \pi_{J'}
\end{align*}
by the definition of $\pi_{J'}$ (see \eqref{eq: def of algebraic piJ}).
In the case (2), we have $i_1=\min J_k -1$ or $i_1=\max J_k +1$. In either case, we have
\begin{align*}
 \pi_{i_1}\pi_J 
 &= \pi_{J_1}\cdots \pi_{J_{k-1}}(\pi_{i_1} \pi_{J_k})\pi_{J_{k+1}}\cdots \pi_{J_m}\\
 &= \pi_{J_1}\cdots \pi_{J_{k-1}}\left( \binom{|J_k|+1}{1}\pi_{J'(i_1)}\right) \pi_{J_{k+1}}\cdots \pi_{J_m}
\end{align*}
by Proposition~\ref{prop: joint lemma 2 in y} and $\binom{|J_k|+1}{1}=\binom{|J_k|+1}{|J_k|}$. Hence, we obtain
\begin{align*}
 \pi_{i_1}\pi_J = \binom{|J_k|+1}{1}\pi_{J'}
\end{align*}
in this case.
In the case (3), we have
\begin{align*}
 \pi_{i_1}\pi_J 
 &= \pi_{J_1}\cdots \pi_{J_{k-1}}(\pi_{J_k}\cdot \pi_{i_1}\cdot \pi_{J_{k+1}})\pi_{J_{k+2}}\cdots \pi_{J_m}\\
 &= \pi_{J_1}\cdots \pi_{J_{k-1}}\left( \binom{|J_k|+1}{|J_k|} \pi_{J_k\cup\{i_1\}}\cdot \pi_{J_{k+1}}\right)\pi_{J_{k+2}}\cdots \pi_{J_m}\\
 &= \pi_{J_1}\cdots \pi_{J_{k-1}}\left( \binom{|J_k|+1}{|J_k|} \binom{|J_k|+1+|J_{k+1}|}{|J_k|+1} \pi_{J_k\cup\{i_1\}\cup J_{k+1}}\right)\pi_{J_{k+2}}\cdots \pi_{J_m}\\
 &= \binom{|J_k|+1}{|J_k|}\binom{|J_k|+1+|J_{k+1}|}{|J_k|+1}\pi_{J'}
\end{align*}
by Proposition~\ref{prop: joint lemma 2 in y} again.
Thus, in either case of (1)-(3), we see that the product $\pi_{i_1}\pi_J$ is an integer multiple of $\pi_{J'}$.

If $J'= J$, then we have $i_1\in J_k$ for some $1\le k\le m$.
In this case, let us write $J_k=[a,b]$ for some $1\le a\le b\le n-1$.
Then, Proposition~\ref{prop: joint lemma in y} implies that
\begin{equation}
\begin{split}\label{eq: final computations 10}
 \pi_{i_1}\pi_J
 &=\pi_{J_1}\cdots \pi_{J_{k-1}}(\pi_{i_1}\pi_{J_k})\pi_{J_{k+1}}\cdots \pi_{J_m}\\
 &=\pi_{J_1}\cdots \pi_{J_{k-1}}\Big((b-i_1+1)\pi_{[a-1,b]}+(i_1-a+1)\pi_{[a,b+1]}\Big)\pi_{J_{k+1}}\cdots \pi_{J_m} \\
 &=(b-i_1+1)\pi_{J_1}\cdots \pi_{J_{k-1}}\pi_{[a-1,b]}\pi_{J_{k+1}}\cdots \pi_{J_m} \\
 &\hspace{50pt}
 + (i_1-a+1)\pi_{J_1}\cdots \pi_{J_{k-1}}\pi_{[a,b+1]}\pi_{J_{k+1}}\cdots \pi_{J_m}.
 \end{split}
\end{equation}
For the first summand in the last expression, we have
\begin{equation}\label{eq: -infty}
\begin{split}
 &\pi_{J_1}\cdots \pi_{J_{k-1}}\pi_{[a-1,b]}\pi_{J_{k+1}}\cdots \pi_{J_m} \\
 &\hspace{20pt}=
 \begin{cases}
  \pi_{J\cup\{a-1\}} \quad &\text{(if $\max J_{k-1}< a-2$)}, \vspace{10pt}\\
  \displaystyle{\binom{|J_{k-1}|+b-a+2}{|J_{k-1}|}\pi_{J\cup\{a-1\}}} &\text{(if $\max J_{k-1}=a-2$)},
 \end{cases}
\end{split}
\end{equation}
where the first case follows from the definition of $\pi_{J\cup\{a-1\}}$, and the second case follows by applying Proposition~\ref{prop: joint lemma 2 in y} to $\pi_{J_{k-1}}\pi_{[a-1,b]}$ in the left hand side.
Here, we take the convention $J_0=\{-\infty\}$ 
since only the first case of \eqref{eq: -infty} holds when $k=1$.
Similarly, for the second summand in the last expression of \eqref{eq: final computations 10}, we have
\begin{equation}\label{eq: +infty}
\begin{split}
 &\pi_{J_1}\cdots \pi_{J_{k-1}}\pi_{[a,b+1]}\pi_{J_{k+1}}\cdots \pi_{J_m}\\
 &\hspace{20pt}=
 \begin{cases}
  \pi_{J\cup\{b+1\}} \quad &\text{(if $\min J_{k+1}> b+2$)}, \vspace{10pt}\\
  \displaystyle{\binom{|J_{k+1}|+b-a+2}{b-a+2}\pi_{J\cup\{b+1\}}} &\text{(if $\min J_{k+1}=b+2$)}.
 \end{cases}
\end{split}
\end{equation}
Again, we take the convention $J_{m+1}=\{+\infty\}$ since only the first case of \eqref{eq: +infty} holds when $k=m$.
Thus, it follows from \eqref{eq: final computations 10} that $\pi_{i_1}\pi_J$ is a linear combination of $\pi_{J\cup\{a-1\}}$ and $\pi_{J\cup\{b+1\}}$.
Applying this result to the right hand side of \eqref{eq: final computations 20}, we see that $f$ is a linear combination of $\pi_J$ for $J\subseteq[n-1]$.

Finally, we consider the case $d>n-1$.
In this case, the number of $\pi_{i_k}$ appearing in \eqref{eq: monomial in truncated e1} is greater than $n-1$.
Hence, we can write
\begin{align*}
 f = \pi_{i_1}\pi_{i_2}\cdots \pi_{i_d} 
 = ( \pi_{i_1}\pi_{i_2}\cdots \pi_{i_{n-1}} ) \cdot (\pi_{i_{n}}\pi_{i_{n+1}}\cdots \pi_{i_d} ).
\end{align*} 
Applying the argument used in the previous case to the product $\pi_{i_1}\pi_{i_2}\cdots \pi_{i_{n-1}}$, we see that it is an integer multiple of $\pi_{[n-1]}$. Since we have
\begin{align*}
\pi_{[n-1]} \cdot \pi_{i_n} = 0
\end{align*} 
by Proposition~\ref{prop: joint lemma in y}, we see that $f=0$ in this case. This completes the proof.
\end{proof}

\bigskip



\begin{thebibliography}{10}
\bibitem{ab-de-ga-ha}
H. Abe, L. DeDieu, F. Galetto, and M. Harada,
\emph{Geometry of Hessenberg varieties with applications to Newton-Okounkov bodies}, 
Selecta Math. (N.S.) \textbf{24} (2018), no. 3, 2129--2163.
\bibitem{ab-ho-ma}
H. Abe, T. Horiguchi, and M. Masuda,
\emph{The cohomology rings of regular semisimple Hessenberg varieties for $h=(h(1),n,\ldots,n)$}, 
J. Comb. \textbf{10} (2019), no. 1, 27--59.
\bibitem{AHMMS}
	T. Abe, T. Horiguchi, M. Masuda, S. Murai, and T. Sato,
	\emph{Hessenberg varieties and hyperplane arrangements},
	J. Reine Angew.\ Math.\ {\bf 764} (2020), 241--286.
\bibitem{AHHM}
H. Abe, M. Harada, T. Horiguchi, and M. Masuda,
\emph{The cohomology rings of regular nilpotent Hessenberg varieties in Lie type A}, 
Int. Math. Res. Not., 
(\textbf{2019}), no.17, 5316--5388.
\bibitem{AHKZ}
H. Abe, T. Horiguchi, H. Kuwata, H. Zeng,
\emph{Geometry of Peterson Schubert calculus in type A and left-right diagrams}, 
arXiv:2104.02914.
\bibitem{an-ty}
D. Anderson and J. Tymoczko, 
\emph{Schubert polynomials and classes of Hessenberg varieties}, 
J. Algebra \textbf{323} (2010), no. 10, 2605--2623. 
\bibitem{Ba}
A. B$\breve{\text{a}}$llibanu, 
	\emph{The Peterson Variety and the Wonderful Compactification}, 
	Represent. Theory \textbf{21} (2017), 132--150.
\bibitem{ChoHongLee1}
S. Cho, J. Hong, and E. Lee,
\emph{Bases of the equivariant cohomologies of regular semisimple Hessenberg varieties}, 
arXiv:2008.12500v2.
\bibitem{De Mari-Procesi-Shayman}
   F. De Mari, C. Procesi and M. A. Shayman,
   \emph{Hessenberg varieties}, 
   Trans. Amer. Math. Soc. 332 (1992),  no. 2, 529-534. 
\bibitem{Dre1}
E. Drellich,
\emph{Combinatorics of equivariant cohomology: Flags and regular nilpotent Hessenberg varieties}, 
PhD thesis, University of Massachusetts, 2015.
\bibitem{fu-ha-ma}
Y. Fukukawa, M. Harada, and M. Masuda, 
\emph{The equivariant cohomology rings of Peterson varieties}, 
J. Math. Soc. Japan \textbf{67} (2015), no. 3, 1147--1159. 
\bibitem{fult97}
W. Fulton, \emph{Young Tableaux}, 
London Mathematical Society Student Texts, 35. Cambridge University Press, Cambridge.
\bibitem{Fulton}
	W. Fulton, 
	\emph{Intersection theory}, 
	Second edition. Results in Mathematics and Related Areas. 3rd Series. A Series of Modern Surveys in Mathematics, 2. Springer-Verlag, Berlin, 1998.
\bibitem{GoGo}
	R. Goldin and B. Gorbutt
	\emph{A positive formula for type $A$ Peterson Schubert calculus}
	arXiv:2004.05959.
\bibitem{GoMiSi}
	R. Goldin, L. Mihalcea, R. Singh,
	\emph{Positivity of Peterson Schubert Calculus}
	arXiv:2106.10372.
\bibitem{GoSi}
	R. Goldin and R. Singh,
	\emph{Equivariant Chevalley, Giambelli, and Monk Formulae for the Peterson Variety},
	arXiv:2111.15663.
\bibitem{ha-ho-ma}
M. Harada, T. Horiguchi, and M. Masuda, \emph{The equivariant cohomology rings of Peterson varieties in all Lie types}, 
Canad. Math. Bull. \textbf{58} (2015), no. 1, 80--90. 
\bibitem{HaTy}
	M. Harada and J. Tymoczko, \emph{A positive Monk formula in the $S^1$-equivariant cohomology of type $A$ Peterson varieties}, 
	Proc. Lond. Math. Soc. (3) \textbf{103} (2011), no. 1, 40--72.
\bibitem{Ho}
	T. Horiguchi, \emph{Mixed Eulerian numbers and Peterson Schubert calculus},
	arXiv:2104.14083v2.
\bibitem{Insko}
	E. Insko,
	\emph{Schubert calculus and the homology of the Peterson variety}, 
	Electron. J. Combin. \textbf{22} (2015), no.2, Paper 2.26, 12 pp.
\bibitem{Klyachko}
   A. Klyachko,
   \emph{Orbits of a maximal torus on a flag space}, 
   Functional Anal. Appl. 19 (1985),  no. 2, 65-66. 
\bibitem{Kostant}
	B. Kostant, 
	\emph{Flag manifold quantum cohomology, the Toda lattice, and the representation with highest weight $\rho$},
	Selecta Math. (N.S.) \textbf{2} (1996), no. 1, 43--91.
\bibitem{Nadeau-Tewari}
P. Nadeau and V. Tewari,
	\emph{The permutahedral variety, mixed Eulerian numbers, and principal specializations of Schubert polynomials},
	Int. Math. Res. Not. (\textbf{2023}), no.5, 3615--3670.
\bibitem{Peterson}
D. Peterson, 
\emph{Quantum cohomology of $G/P$},
Lecture course, M.I.T., Spring term, 1997. 
\bibitem{Rietsch06}
	K. Rietsch, 
	\textit{A mirror construction for the totally nonnegative part of the Peterson variety},
	Nagoya Math. J. \textbf{183} (2006), 105--142.
\bibitem{ste94}
J. Stembridge, 
\textit{Some permutation representations of Weyl groups associated with the cohomology of toric varieties},
Adv. Math. \textit{106} (1994), no. 2, 244--301.
\bibitem{so-ty}  
E. Sommers and J. Tymoczko,
\emph{Exponents for B-stable ideals}, 
Trans. Amer. Math. Soc. \textbf{358} (2006), no. 8, 3493--3509. 
\bibitem{Teff11}
N. Teff, 
\emph{Representations on Hessenberg varieties and Young's rule}, 
23rd International Conference on Formal Power Series and Algebraic Combinatorics (FPSAC 2011), 903--914, Discrete Math. Theor. Comput. Sci. Proc., AO, Assoc. Discrete Math. Theor. Comput. Sci., Nancy, 2011.
\bibitem{ty}
J. Tymoczko, 
\emph{Linear conditions imposed on flag varieties}, 
Amer. J. Math. \textbf{128} (2006), no. 6, 1587--1604. 
\bibitem{Ty3}
J. Tymoczko,
\emph{Permutation Actions on Equivariant Cohomology of Flag Varieties},
Toric Topology, 365--84, Contemp. Math., \textbf{460}, Amer. Math. Soc., Providence, RI, 2008.
\end{thebibliography}
\end{document}